\documentclass[review,onefignum,onetabnum]{siamart190516}

\usepackage{lipsum}
\usepackage{amsmath,amssymb,amsfonts,latexsym,stmaryrd}
\usepackage{graphicx,float}
\usepackage{mathrsfs} 
\usepackage{color}
\usepackage{epstopdf}
\usepackage{algorithmic}
\usepackage{multirow}
\ifpdf
  \DeclareGraphicsExtensions{.eps,.pdf,.png,.jpg}
\else
  \DeclareGraphicsExtensions{.eps}
\fi

\def\O{\Omega}

\def\g{\gamma}

\def\l{\lambda}

\def\VK{V^{\E}}

\renewcommand\sp{\mathop{\mathrm{Sp}}\nolimits}

\usepackage{booktabs}
\usepackage{float}

\newcommand\bu{\boldsymbol{u}}

\def\Vh{V_h}

\def\PiK{\Pi^{\E}}
\def\CM{\mathcal{X}}
\def\CN{\mathcal{Y}}

\renewcommand\sp{\mathop{\mathrm{sp}}\nolimits}

\def\CT{\mathcal{T}}
\def\CM{\mathcal{X}}
\def\CN{\mathcal{Y}}

\def\O{\Omega}
\def\P{\mathbb{P}}
\def\PiK{\Pi^{\nabla,E}}

\def\Pio{\Pi^{E}}
\def\Vh{V_h}
\def\VK{V^{E}_h}
\def\WK{\widetilde{V}_h^E}
\def\l{\lambda}
\def\g{\gamma}

\renewcommand\div{\mathop{\mathrm{div}}\nolimits}
\renewcommand\div{\mathop{\mathrm{div}}\nolimits}

\def\CT{{\mathcal T}}

\newcommand\R{\mathbb{R}}

\renewcommand\H{\mathrm{H}}
\renewcommand\L{\mathrm{L}}

\renewcommand\O{\Omega}

\renewcommand\div{\mathop{\mathrm{div}}\nolimits}

\renewcommand\sp{\mathop{\mathrm{sp}}\nolimits}

\newcommand{\vertiii}[1]{{\left\vert\kern-0.25ex\left\vert\kern-0.25ex\left\vert #1 
    \right\vert\kern-0.25ex\right\vert\kern-0.25ex\right\vert}}

\nolinenumbers

\newsiamremark{remark}{Remark}
\newsiamremark{hypothesis}{Hypothesis}
\crefname{hypothesis}{Hypothesis}{Hypotheses}
\newsiamthm{claim}{Claim}

\headers{VEM allowing small edges for the acoustic problem}{D. Amigo, F. Lepe and  G. Rivera}

\title{VEM allowing small edges for the acoustic problem\thanks{Submitted to the editors DATE.
\funding{DA and FL were partially supported by DIUBB through project 2120173 GI/C Universidad del B\'io-B\'io. FL was partially supported by 
ANID-Chile through FONDECYT project 11200529 (Chile). GR was supported by Universidad de Los Lagos Regular R02/21 and ANID-Chile through FONDECYT project 1231619 (Chile).
}}}

\author{Danilo Amigo\thanks{GIMNAP-Departamento de Matem\'atica, Universidad del B\'io-B\'io, Casilla 5-C, Concepci\'on, Chile. 
\texttt{danilo.amigo2101@alumnos.ubiobio.cl}.}
\and
Felipe Lepe\thanks{GIMNAP-Departamento de Matem\'atica, Universidad del B\'io-B\'io, Casilla 5-C, Concepci\'on, Chile. 
\texttt{flepe@ubiobio.cl}.}
\and
Gonzalo Rivera\thanks{Departamento de Ciencias Exactas, Universidad de Los Lagos, Osorno, Chile.
\texttt{gonzalo.rivera@ulagos.cl}}}

\usepackage{amsopn}

\ifpdf
\hypersetup{
  pdftitle={VEM allowing small edges for the acoustic problem},
  pdfauthor={Danilo Amigo, Felipe Lepe, Gonzalo Rivera}
}
\fi

\begin{document}
\maketitle
\nolinenumbers

\begin{abstract}
In this paper we propose and analyze a virtual element method to approximate the natural frequencies of the acoustic eigenvalue problem with
polygonal meshes that allow the presence of small edges. With the aid of a suitable seminorm that depends on the stabilization of the small edges method, we prove
convergence and error estimates for the eigenfrequencies and eigenfunctions of the problem, supporting our analysis on the compact operators theory. We report some numerical tests that allows us to assess the performance of the method and the accuracy on the approximation.
\end{abstract}



\section{Introduction}
\label{sec:introduccion}
Numerical methods for the acoustic problem has been a matter of study from several years due the importance of the knowledge of the natural frequencies of
fluids with acoustic properties. The acoustic phenomenon is well established from the physical point of view, where in \cite{Soize1997StructuralAA} it is possible to 
find a complete description of the meaning of the physical interpretation of the acoustic system when we are in presence of fluids that may produce for instance,  internal dissipation, or inviscid fluids. When dissipative fluids are considered, naturally, the eigenvalue problem that emerge is non-linear and the analysis for this type of problem is not direct. We refer to \cite{MR1770352,MR3854050} where it is possible to find the functional treatment of such problems. On the other hand, fluids that not have presence of dissipative properties lead to linear eigenvalue problems, where the literature related to the numerical approximation of the solution of the acoustic system is abundant, not only on what concerns to  numerical methods, but also on the different formulations of the systems of partial differential equations.

Regarding to the virtual element method (VEM), the applications to solve eigenvalue problems are well documented on the literature for different
nature of partial differential equations and the spectral problems associated to them. We refer, for instance, to \cite{gardini2, MR3867390, GMV2018, MR4229296,MR4253143,MR3340705, MR4050542, MVsiam2021} where the VEM have shown, on its conforming and non-conforming versions,  the accuracy in the approximation of the solutions of eigenvalue problems related to second and fourth order eigenvalue problems, elasticity and Stokes eigenvalue problems, mixed formulations, etc. In particular, we refer to  \cite{BMRR,MR4550402} where the displacement formulation of the acoustic eigenvalue problem has been considered, involving VEM spaces to discretize the space $\H(\div)$. All these references and the references therein, operate under the classic assumptions of \cite{BBCMMR2013} for the polygonal meshes which consists in star-shaped polygons and the sides of these polygons are not allowed to be too small. This last condition has been relaxed on 
\cite{BLR2017} for the two dimensional Laplace source problem and extended in \cite{MR3815658} for polyhedral allowing small faces leading to an important advance on the study of VEM but, for the best of our knowledge, is only available for second order elliptic problems involving the discretization of $\H^1$ spaces. In this sense, the new approach of small edges have emerged as an excellent tool to approximate the solutions of eigenvalue problems as is presented in \cite{danilo_eigen, MR4284360, LR3}. This is precisely what motivates our work in order to continue with our research program on the applications of the VEM allowing for small edges on spectral problems. In the present case we focus on the acoustic eigenvalue problem in its pressure formulation. Despite to the fact that in \cite{MR4284360} the Laplace operator has been already studied with the VEM allowing small edges as in our present case, we have to precise some difference: in one hand, we need to stabilize every bilinear form for the acoustic problem which for the Steklov problem is no needed. On the other hand, since the left-hand side of our discrete problem is not stable, it is not possible to relate the $\H^1$ norm and the seminorm introduced in  \cite{BLR2017}, implying  that the convergence analysis in $\H^1$ norm for the solution operators must be analyzed with different techniques compared with \cite{MR4284360}, where now, the discrete coercivity is no longer needed for the analysis.  Let us remark that the choice of the pressure formulation for the acoustic problem  is precisely since the variational formulation demands to seek the pressure on the space $\H^1$ and according to the classic regularity for the pure-Neumann Laplace problem,  the small edges approach can be used according to \cite{BLR2017}. Moreover, since the elasticity equations on its source and spectral problems have been already studied with the VEM allowing small edges (see \cite{MR4581469,danilo_eigen}), the  analysis of the acoustic problem with a VEM allowing for small edges  opens the gate to analyze numerically a more challenging problem as the elastoacustic problem which we can describe as a formulation depending on the displacement of the solid and the pressure of the fluid (see \cite{MR1993937} for instance). 

The outline of our papers is as follows: In Section \ref{sec:model} we present the problem under consideration. This includes the bilinear forms, functional spaces, regularity of the eigenfunctions, the solution  operator and the corresponding  spectral characterization. Section \ref{sec:virtual} states the virtual element method, where the definitions and assumptions on the mesh are presented. With the aim of develop a small edges method, we introduce a suitable norm depending on the stabilization term, which is taking in an  appropriate way for the small edges scheme. With the discrete bilinear forms we introduce the discrete eigenvalue problem and the discrete solution operator in order to perform the analysis of well posedness, spectral convergence and the derivation of error estimates. We end the paper in Section \ref{sec:numerics} reporting a series of numerical tests to assess the performance of the method, in order to confirm the theoretical results.


 \section{The  model problem}
 \label{sec:model}
\label{sec:pressure}
Let $\O$ be an open and bounded bidimensional domain with Lipschitz boundary $\partial\O$. The classic acoustic problem is:  Find $\omega\in\mathbb{R}$, the displacement $\bu$ and the pressure $p$ on a domain $\Omega\subset\mathbb{R}^{\texttt{d}}$, such that 
\begin{equation}\label{def:acustica}
\left\{
\begin{array}{rcll}
\nabla p-\omega^2\rho \bu & = & \boldsymbol{0}&\text{in}\,\O\\
p+\rho c^2\div\bu & = & 0 &\text{in}\,\O\\
\bu\cdot\boldsymbol{n}&=&0&\text{on}\,\partial\O,
\end{array}
\right.
\end{equation}
where $\rho$ is the density, $c$ is the sound speed, and $\boldsymbol{n}$ is the outward unitary vector. 
Now, using the second equation of \eqref{def:acustica} we can eliminate the displacement in order to obtain a problem depending only on the pressure. This problem consists into find the pressure $p$ and the frequency $\omega$ such that 
\begin{equation}\label{def:acustica_pressure}
\left\{
\begin{aligned}
 c^2\div \left(\frac{1}{\rho}\nabla p\right)+\frac{\omega^2}{\rho} p& = &0&\text{ in }\,\O,\\
\nabla p\cdot\boldsymbol{n}&=&0&\text{ on }\,\partial\O,
\end{aligned}
\right.
\end{equation}

 A variational formulation for \eqref{def:acustica_pressure} is: Find $\omega\in\mathbb{R}$ and $0\neq p\in \H^1(\O)$ such that
\begin{equation*}
\displaystyle c^2\int_{\O}\frac{1}{\rho}\nabla p\cdot\nabla v=\omega^2\int_{\O}\frac{1}{\rho}pv \quad\forall v\in \H^1(\O).
\end{equation*}
Let us define the bilinear forms $a:\H^1(\O)\times \H^1(\O)\rightarrow\mathbb{R}$ and $b:\H^1(\O)\times \H^1(\O)\rightarrow\mathbb{R}$, which are given by
\begin{equation*}
a(q,v):=c^2\int_{\O}\frac{1}{\rho}\nabla q\cdot\nabla v\quad \text{and}\quad b(q,v):=\int_{\O}\frac{1}{\rho}qv\quad\forall q,v \in\H^1(\O).
\end{equation*}
With a shift argument and setting $\lambda := \omega^{2}+1$, we arrive to the following problem: Find $\lambda\in\mathbb{R}$ and $0\neq p\in \H^1(\O)$ such that
\begin{equation}
\label{eq:pression}
\widehat{a}(p,v)=\lambda b(p,v)\quad\forall v\in \H^1(\O),
\end{equation}
where the bilinear form $\widehat{a}:\H^1(\O)\times \H^1(\O)\rightarrow\mathbb{R}$ is defined for all $q,v\in \H^1(\O)$ by 
$$ \widehat{a}(q,v) := a(q,v) + b(q,v).$$

It is easy to check that $\widehat{a}(\cdot,\cdot)$ is coercive in $\H^1(\O)$. This allows us to introduce the solution operator $T: \H^1(\O)\rightarrow \H^1(\O)$, defined by 
$Tf=\widetilde{p}$, where $\widetilde{p}\in \H^1(\O)$ is the solution of the corresponding  associated source problem: Find $\widetilde{p}\in\H^1(\O)$ such that
\begin{equation}
\label{eq:source_pr}
\widehat{a}(\widetilde{p},v)=b(f,v)\quad\forall v\in\H^1(\O).
\end{equation}

The regularity results that we need for our purposes are the ones related to the Laplace problem with pure null boundary conditions. This regularity is stated in the following lemma (see \cite[Lemma 2.2]{MR3340705} and \cite{MR0775683}).
\begin{lemma}
\label{lmm:regularity}
There exists $r_{\Omega}>1/2$ such that the following results hold:
\begin{enumerate}
\item For all $f\in\H^1(\O)$ and for all $r\in (1/2,r_{\O})$ the solution $\widehat{p}$ of \eqref{eq:source_pr} satisfies $\widehat{p}\in\H^{1+s}(\O)$ with $s:=\min\{r,1\}$. Moreover, there exists a constant $C>0$ such that
\begin{equation*}
\|\widetilde{p}\|_{1+s}\leq C\|f\|_{1,\O};
\end{equation*}
\item If $p$ is an eigenfunction of problem \eqref{eq:pression} with eigenvalue $\lambda$, for all $r\in(1/2,r_{\O})$ there hold that $p\in\H^{1+r}(\O)$ and also, there exists a constant $C>0$, depending on $\lambda$, such that
\begin{equation*}
\| p\|_{1+r}\leq C\|p\|_{1,\O}.
\end{equation*}
\end{enumerate} 
\end{lemma}

In virtue of Lemma \ref{lmm:regularity}, the solution operator $T$  results to be compact due the compact inclusion of $\H^{1+s}(\O)$ onto $\H^1(\O)$ and self-adjoint with respect to $\widehat{a}(\cdot,\cdot)$. We observe that $(\lambda,p) \in \mathbb{R}\times \H^1(\O)$ solves \eqref{eq:pression} if and only if $(\mu,p) \in \mathbb{R}\times \H^1(\O)$ is an eigenpair of $T$, with $\mu := 1/\lambda$. Finally, since we have the additional regularity for the eigenfunctions, the following spectral characterization of $T$ holds.
\begin{lemma}[Spectral Characterization of $T$]
The spectrum of $T$ satisfies $\sp(T) = \{0,1\} \cup \{\mu_k\}_{k\in\mathbb{N}}$, where $\{\mu_k\}_{k\in\mathbb{N}}$ is a sequence of real and positive eigenvalues that converge to zero, according to their respective multiplicities.
\end{lemma}


\section{The virtual element method}
\label{sec:virtual}
Let us now introduce the ingredients to establish the virtual element method  for the eigenvalue problem \eqref{eq:pression}. First we recall the mesh construction and the assumptions considered in \cite{BBCMMR2013} for the virtual element
method. 
Let $\left\{\CT_h\right\}_h$ be a sequence of decompositions of $\Omega$ into polygons which we denote by $E$. Let us denote by $h_E$   the diameter of the element $E$ and $h$ the maximum of the diameters of all the elements of the mesh, i.e., $h:=\max_{E\in\Omega}h_E$.  Moreover, for simplicity, in what follows we assume that $\kappa $ and $\gamma$  are  piecewise constant with  respect to the decomposition $\mathcal{T}_h$, i.e., they are  piecewise constants for all $E\in \mathcal{T}_h$ (see for instance \cite{BLR2017}).

 For the analysis of the VEM, we will make as in \cite{BBCMMR2013} the following
assumption: there exists a positive real number $\rho$ such that, for every $ E \in\CT_{h}$ and for every $\CT_{h}$,
\begin{itemize}
\item \textbf{A1.} For all meshes
$\CT_h$, each polygon $E\in\CT_h$ is star-shaped with respect to a ball
of radius greater than or equal to $\rho h_{E}$.
\end{itemize}

For any simple polygon $ E$ we define 
\begin{align*}
	\widetilde{V}_h^E:=\{v_h\in \H^1(E):\Delta v_h \in \mathbb{P}_1(E),
	v_h|_{\partial E}\in C^0(\partial E), v_h|_{e}\in \mathbb{P}_1(e) \  \forall e \in \partial E \}.
\end{align*}

Now, in order to choose the degrees of freedom for $\widetilde{V}_h^E$ we define
\begin{itemize}
	\item $\mathcal{V}_E^h$: the value of $w_{h}$ at each vertex of $E$,
\end{itemize}
as a set of linear operators from $\widetilde{V}_h^E$ into $\R$. In \cite{AABMR13} it was established that $\mathcal{V}_E^h$ constitutes a set of degrees of freedom for the space $\widetilde{V}_h^E$.

 On the other hand, we define  the projector $\PiK:\ \WK\longrightarrow\P_1(E)\subseteq\WK$ for
each $v_{h}\in\WK$ as the solution of 
\begin{equation*}
\int_E (\nabla\PiK v_{h}-\nabla v_{h})\cdot\nabla q_{1}=0
\quad\forall q_{1}\in\P_1(E),\qquad \overline{\PiK v_{h}}=\overline{v_{h}},
\end{equation*}
where  for any sufficiently regular
function $v$, we set $\overline{v}:=\vert\partial E\vert^{-1}\int_{\partial E}v.$
We observe that the term $\PiK v_{h}$ is well defined and computable from the degrees of freedom  of $v$ given by $\mathcal{V}_E^h$, and in addition the projector $\PiK$ satisfies the identity  $\PiK(\P_{1}(E))=\P_{1}(E)$ (see for instance \cite{AABMR13}).

We are now in position  to introduce our local virtual space
\begin{equation*}\label{Vk}
\VK
:=\left\{v_{h}\in 
\WK: \displaystyle \int_E \PiK v_{h}p_{1}=\displaystyle \int_E v_{h}p_{1},\quad \forall p_{1}\in 
\mathbb{P}_1(E)\right\}.
\end{equation*}
Now, since $\VK\subset \WK$, the operator $\PiK$ is well defined on $\VK$ and computable  only on the basis of the output values of the operators in $\mathcal{V}_E^h$.
In addition, due to the particular property appearing in definition of the space $\VK$, it can be seen that $\forall p_{1} \in \mathbb{P}_1(E)$ and $\forall v_h\in \VK$ the term $(v_h,p_{1})_{0,E}$
is computable from $\PiK v_h$, and hence  the  ${\mathrm L}^2(E)$-projector operator $\Pio: \VK\to \P_1(E)$ defined  by
$$\int_E \Pi^{E}v_h p_{1}=\int_E v_h p_{1}\qquad \forall p_{1}\in \mathbb{P}_1(E),$$
depends only on the values of the degrees of freedom of $v_h$.  Actually, it is easy to check that the projectors $\PiK$ and $\Pi^{E}$ are the same operators  on the space $\VK$ (see \cite{AABMR13} for further details).

Finally, for every decomposition $\CT_h$ of $\Omega$ into simple polygons $ E$ we define the global virtual space
\begin{equation}
\label{eq:globa_space}
\Vh:=\left\{v\in \H^1(E):\ v|_{ E}\in\VK\quad\forall E\in\CT_h\right\},
\end{equation}
and the global degrees of freedom are obtained by collecting the local ones, with the nodal and interface degrees of freedom corresponding to internal entities counted only once
those on the boundary are fixed to be equal to zero in accordance with the ambient space $\H_{0}^{1}(\Omega)$.
\subsection{Discrete bilinear forms}
In order to propose the discrete counterparts of $a(\cdot,\cdot)$ and $b(\cdot,\cdot)$, we split these  forms as follows
\begin{equation*}
a(q,v) = \displaystyle{\sum_{E \in \mathcal{T}_{h}} a^{E}(q,v)}, \quad b(q,v) = \displaystyle{\sum_{E \in \mathcal{T}_{h}} b^{E}(q,v)}.
\end{equation*}

Now, for each polygon $E \in \mathcal{T}_{h}$, we introduce the following symmetric and semi-positive definite bilinear form $S^{E} : V_{h}\times V_{h} \rightarrow \mathbb{R}$ as follows:
\begin{equation*}
S^{E}(q_{h},v_{h}) := h_{E}\displaystyle{\int_{\partial E} \partial_{s}q_{h}\partial_{s}v_{h}} \quad \forall q_{h}, v_{h} \in V_{h}^{E},
\end{equation*}
where $\partial_{s}$ denotes a derivative along the edge. Then, we define the local discrete bilinear form $a_{h}^{E}(\cdot,\cdot) : V_{h}\times V_{h} \rightarrow \mathbb{R}$ by
\begin{equation*}
a_{h}^{E}(q_{h},v_{h}) := a^{E}(\PiK q_{h}, \PiK v_{h}) + S^{E}(q_{h} - \PiK q_{h}, v_{h} - \PiK v_{h}),
\end{equation*}
for all $q_{h}, v_{h} \in V_{h}^{E}$. 

Now, we introduce the following discrete semi-norm (see \cite{BLR2017} for details)
\begin{equation*}
\label{eq:triple}
\vertiii{v}_{E}^2:=a^{E}\big(\PiK v,\PiK v)+S^{E}(v-\bar{v},v-\bar{v})\qquad\forall
v\in\VK+\mathcal{V}^{E},
\end{equation*}
where  $\mathcal{V}^{E}\subseteq \H^1(E)$ is a subspace of sufficiently
regular functions for $S^{E}(\cdot,\cdot)$ to make sense.
 
For any sufficiently regular functions,
we introduce the following global semi-norms
\begin{equation*}
\vertiii{v}^2:=\sum_{E\in\CT_h}\vertiii{v}_{E}^2,\qquad
\left|v\right|_{1,h}^2
:=\sum_{E\in\CT_h}\left\|\nabla v\right\|_{0,E}^2.
\end{equation*}
According to  \cite[Lemma~3.1]{BLR2017} the existence of  positive constants $C_1,C_2, C_3$, independent of $h$ but depending on the polygon $E$, such that
\begin{align}
C_1\vertiii{v}_{E}^2\le a_h^{E}(v,v)\le C_2\vertiii{v}_{E}^2\quad\forall v\in\VK,\label{eqrefgt1}\\
a_h^{E}(v,v)\le C_3(\vertiii{v}^2+\vert v\vert_{1,E}^2)\quad\forall v\in\VK.\label{eqrefgt2}
\end{align}
In addition, it holds
\begin{align}
a^{E}(v,v)\le C_4\vertiii{v}_{E}^2\quad\forall v\in\VK,\label{eqrefgt3}\\
\vertiii{p_{1}}_{E}^2\le C_5 a^{E}(p_{1},p_{1})\quad\forall p_{1}\in\P_1(E),\label{eqrefgt4}
\end{align}
where $C_4$ and $C_5$ are positive constants independent of $h$ but depending on the polygon $E$. On the other hand, to introduce the local discrete counterpart of $b^{E}(q_{h},v_{h})$, we consider any symmetric and semi-positive definite bilinear form $S_{0}^{E} : V_{h}^{E}\times V_{h}^{E} \rightarrow \mathbb{R}$ satisfying
\begin{equation*}
b_{0}b^{E}(v_{h},v_{h}) \leq S_{0}^{E}(v_{h},v_{h}) \leq b^{1}b^{E}(v_{h},v_{h}) \quad \forall v_{h} \in V_{h}^{E},
\end{equation*}
where $b_{0},b^{1}$ are two   positive constants. Then, we define for each polygon $E$ the local (and computable) bilinear form $b_{h}^{E} : V_{h}^{E}\times V_{h}^{E} \rightarrow \mathbb{R}$ by
\begin{equation*}
b_{h}^{E}(q_{h},v_{h}) = b^{E}(\Pio q_{h}, \Pio v_{h}) + S_{0}^{E}(q_{h} - \Pio q_{h}, v_{h} - \Pio v_{h})\quad \forall q_{h}, v_{h} \in V_{h}^{E}.
\end{equation*}
We remark that the discrete bilinear form $b_{h}^{E}(\cdot,\cdot)$ satisfies the classical properties of consistency and stability. Then, the global discrete bilinear forms $a_{h}(\cdot,\cdot)$ and $b_{h}(\cdot,\cdot)$ are be expressed componentwise as follows
 \begin{equation*}
\label{eq:bilineal_form_B_split_{h}}
\begin{split}
a_{h}(q_{h},v_{h}):
= \sum_{ E\in\CT_h} a_{h}^{E}( q_{h}, v_{h}), \quad
b_{h}(q_h,v_h) := \sum _{E\in\CT_h} b_{h}^{E}(q_h,v_h), \\
\widehat{a}_{h}(q_{h},v_{h}) := \sum_{ E\in\CT_h} a_{h}^{E}( q_{h}, v_{h}) + b_{h}^{E}(q_h,v_h).
\end{split}
\end{equation*}
\subsection{Spectral discrete problem}
Now we introduce the VEM discretization of problem \eqref{eq:pression}. To do this task, we require the global space $V_h$
defined in \eqref{eq:globa_space} together with the assumptions introduced in Section \ref{sec:virtual}.

Setting $\lambda_{h} := \omega_{h}^{2}+1$, the spectral problem reads as follows: Find $\lambda_h\in\mathbb{R}$ and $0\neq p_h\in V_h$ such that
\begin{equation}
\label{eq:spectral_disc}
\widehat{a}_{h}(p_h,v_h)=\lambda_{h} b_h(p_h,v_h) \quad \forall v_h \in V_h.
\end{equation}

It is possible to prove that $\widehat{a}_h(\cdot,\cdot)$ is $V_h$-coercive. Indeed, for $v_h\in V_h$, using \eqref{eqrefgt1} and \eqref{eqrefgt3} we have
\begin{multline*}
\widehat{a}_h(v_h,v_h)=\sum_{E\in\CT_h}a_h^E(v_h,v_h)+b_h^E(v_h,v_h) \\
\geq \sum_{E\in\CT_h}C_1\vertiii{v_h}^2_E + \min\{b_{0}(E),1\}b^E(v_h,v_h) \\
\geq \min\left\lbrace \min_{E\in\CT_h}\{ C_1(E)C_4(E)^{-1}\}, b_{0}(E),1\right\rbrace \sum_{E\in\CT_h}a^E(v_h,v_h) + b^E(v_h,v_h)\geq \underline{C}\|v_h\|_{1,\O}^2,
\end{multline*}
where $\displaystyle \underline{C}:=C(\rho,c)\min\left\lbrace \min_{E\in\CT_h}\{ C_1(E)C_4(E)^{-1}\}, b_{0}(E),1\right\rbrace$ and $C(\rho,c)$ is a positive constant depending on the density and  sound speed of the fluid. Moreover, thanks to \eqref{eqrefgt3}, we obtain for every $w_h \in V_h$:
\begin{multline}\label{eq:stabbb}
\vertiii{w_{h}}^{2} + \|w_{h}\|_{0,\O}^{2} \geq \sum_{E\in\CT_h} C_{4}^{-1}a^{E}(w_h,w_h) + \|w_{h}\|_{0,E}^{2} \\
\geq c^{2}\rho^{-1}\min_{E\in\CT_h}\{C_{4}(E)^{-1},1\} \left(|w_{h}|_{1,\O}^{2} + \|w_{h}\|_{0,\O}^{2}\right) = C_{1}^{*}\|w_{h}\|_{1,\O}^{2},
\end{multline}
where $\displaystyle C_{1}^* := c^{2}\rho^{-1}\min_{E\in\CT_h}\{C_{4}(E)^{-1},1\}$. 

On the other hand, the coercivity of $\widehat{a}_{h}(\cdot,\cdot)$ allows us to introduce the discrete solution operator $T_h: V_h\rightarrow V_h$ defined by $f_h\mapsto T_h f_h:=\widetilde{p}_h$, where $\widetilde{p}_h\in V_h$ is the solution of the discrete load problem 
\begin{equation*}
\label{eq:source_disc}
\widehat{a}_{h}(\widetilde{p}_h,v_h)= b_h(f_h,v_h) \quad \forall v_h \in V_h.
\end{equation*}

Let us remark that $T_h$ is well defined due to the Lax-Milgram's lemma and self-adjoint with respect to $\widehat{a}_h(\cdot,\cdot)$.  Moreover, it is easy to check that $(\lambda_h,p_h)\in\mathbb{R}\times V_h$ solves \eqref{eq:spectral_disc} if and only if $(\mu_h,p_h)\in\mathbb{R}\times V_h$ is an eigenpair of $T_h$ with $\mu_h:=1/\lambda_h$. 

Finally, we present the spectral characterization of $\textbf{T}_{h}$.
\begin{lemma}[Spectral characterization of $T_{h}$]\label{spectchar}
The spectrum of $T_{h}$ consists in  $M_{h} := \dim(V_{h})$ positive and real eigenvalues  with a certain multiplicity. 
\end{lemma}

On the other hand, we also have the following well known approximation result for polynomials in star-shaped domain (see for instance \cite{BS-2008}).
\begin{lemma}
\label{eq:polyapprox}
If the assumption {\bf A1} is satisfied, then there exists a constant
$C$, depending only on $k$ and $\g$, such that for every $\widetilde{s}$ with 
$0\le \widetilde{s}\le k$ and for every $v\in\H^{1+\widetilde{s}}(E)$, there exists
$v_{\pi}\in\P_k(E)$ such that
$$
\left\|v-v_{\pi}\right\|_{0,E}
+h_{E}\left|v-v_{\pi}\right|_{1,E}
\le Ch_{E}^{1+\widetilde{s}}\left\|v\right\|_{1+\widetilde{s},E}.
$$
\end{lemma}

Finally, we have the following result, that provides the existence of an interpolant operator on the virtual space (see \cite[Proposition 4.2]{MR3340705}).
\begin{lemma}
\label{eq:interpolant}
Under the assumption  {\bf A1}, for
each $\widetilde{s}$ with $0<\widetilde{s}\le 1$, there exist $\widehat{\sigma}$ and a constant $C$, depending only
on $k$, such that for every $v\in \H^{1+\widetilde{s}}(\O)$, there exists
$v_I\in\Vh$ that satisfies
\begin{align*}
\left|v-v_I\right|_{1+t,E}&\le Ch_{E}^{\widetilde{s}-t}\left|v\right|_{1+\widetilde{s},E}\qquad 0\leq t\leq\min\{\widehat{\sigma},\widetilde{s}\},\\ 
\left\|v-v_I\right\|_{0,E}
&\le Ch_{E}\left|v\right|_{1+\widetilde{s},E}.
\end{align*}
\end{lemma}

As a direct consequence of the above two lemmas, standard results on spectral approximation can be used (see \cite{MR2652780,MR0203473}). Observe that the operator $T_{h}$ is not well defined for any source $f \in \H^1(\O)$ since $b_{h}(\cdot,\cdot)$ is a stabilized bilinear form. This implies that the classical theory of compact operators cannot be employed directly. Inspired in \cite{MR4050542}, in order to fix it and taking adventage of the compactness of $T$, we introduce the operator $P_{h} : \L^{2}(\O) \rightarrow V_{h} \hookrightarrow \H^1(\O)$ defined by the following property: $b(P_{h}u - u, v_{h}) = 0$ for every $v_{h} \in V_{h}$. Is easy to check that $\|P_{h}u\|_{0,\O} \leq \|u\|_{0,\O}$. Now, we define the operator $\widehat{T}_{h} : \H^1(\O) \rightarrow V_{h}$, which is well-defined for any source $f \in \H^1(\O)$. Moreover, the spectra of $T_{h}$ and $\widehat{T}_{h}$ coincide, and the same for the eigenfunctions of $T_{h}$ and $\widehat{T}_{h}$.   The first task  is to prove the convergence in norm of $\widehat{T}_{h}$ to $T$. We begin with the following result.
\begin{lemma}
\label{lmm:conv_norm}
There exists a constant $C>0$ independent of $h$  such that for all $f\in \H^1(\O)$, the following estimate holds
\begin{equation*}
\|(T - \widehat{T}_h)f\|_{1,\O}\leq \mathcal{C}h^s\|f\|_{1,\O},
\end{equation*}
where $s$ is the regularity index given by Lemma \ref{lmm:regularity}.
\end{lemma}
\begin{proof}
Let $f\in \H^1(\O)$ be such that $\widetilde{p}:=Tf$ and $\widetilde{p}_h:=\widehat{T}_hf$. Let $\widetilde{p}_I\in V_h$ be the interpolant of $\widetilde{p}\in \H^1(\O)$ given by Lemma \ref{eq:interpolant}. From the triangle inequality we have
\begin{equation*}
\|(T-\widehat{T}_h)f\|_{1,\O}=\|\widetilde{p}-\widetilde{p}_{h}\|_{1,\O}\leq \|\widetilde{p}-\widetilde{p}_I\|_{1,\O}+\|\widetilde{p}_I-\widetilde{p}_h\|_{1,\O}. 
\end{equation*}
We observe that the first term on the inequality above is immediately controlled by using Lemma \ref{eq:interpolant}, obtaining $ \|\widetilde{p}-\widetilde{p}_I\|_{1,\O}\leq Ch^s\|f\|_{1,\O}$. For the second term, we invoke the $V_h$-coercivity of $\widehat{a}_h(\cdot,\cdot)$, defining $v_h:=\widetilde{p}_I-\widetilde{p}_h$ and using \eqref{eqrefgt1} we have 
\begin{multline*}
\vertiii{v_h}^{2} + \|v_{h}\|_{0,\O}^{2} \leq \sum_{E\in\CT_h} C_{1}^{-1}a_{h}^{E}(v_h,v_h) + b^{E}(v_h,v_h) \\
\leq C\max_{E\in\CT_h}\{C_{1}(E)^{-1},1\}\widehat{a}_h(v_h,v_h) = \max_{E\in\CT_h}\{C_{1}(E)^{-1},1\}\left(\widehat{a}_h(\widetilde{p}_I, v_h) - \widehat{a}_h(\widetilde{p}_h,v_h)\right) \\
= \max_{E\in\CT_h}\{C_{1}(E)^{-1},1\}\left(\underbrace{\widehat{a}_h(\widetilde{p}_I-\widetilde{p}_\pi,v_h) - \widehat{a}(\widetilde{p}-\widetilde{p}_\pi,v_h)}_{\textrm{(I)}} + \underbrace{b(f,v_h) - b_h(P_hf,v_h)}_{\textrm{(II)}}\right),
\end{multline*}
where in the last equality we have used the consistency property for $\widehat{a}_{h}(\cdot,\cdot)$. Now, we need to estimate the contributions on the right-hand side. For \textrm{(I)}, using triangle inequality, \eqref{eqrefgt1} and Cauchy-Schwarz inequality, we obtain 
\begin{multline*}
\textrm{(I)} \leq \displaystyle{\sum_{E\in\CT_h} |\widehat{a}_h^E(\widetilde{p}_I - \widetilde{p}_\pi, v_h) - \widehat{a}^E(\widetilde{p}-\widetilde{p}_\pi,v_h)|} \\
\leq \sum_{E\in\CT_h} |a_h^E(\widetilde{p}_I - \widetilde{p}_\pi, v_h) - a^E(\widetilde{p}-\widetilde{p}_\pi,v_h)| + |b_h^E(\widetilde{p}_I - \widetilde{p}_\pi, v_h) - b^E(\widetilde{p}-\widetilde{p}_\pi,v_h)| \\
\leq \sum_{E\in\CT_h} c^{2}\rho^{-1}C_2(E)\vertiii{\widetilde{p}_I-\widetilde{p}_\pi}_E\vertiii{v_h}_E + c^{2}\rho^{-1}|\widetilde{p}-\widetilde{p}_\pi|_{1,E}|v_h|_{1,E} \\ 
+ \rho^{-1}b_{1}\|\widetilde{p}_I-\widetilde{p}_\pi\|_{0,\O}\|v_{h}\|_{0,\O} + \rho^{-1}\|\widetilde{p}-\widetilde{p}_\pi\|_{0,\O}\|v_{h}\|_{0,\O} \\
\leq \sum_{E\in\CT_h} c^{2}\rho^{-1}C_2(E)\vertiii{\widetilde{p}_I-\widetilde{p}_\pi}_E\vertiii{v_h}_E + c^{2}\rho^{-1}\sqrt{C_{4}}|\widetilde{p}-\widetilde{p}_\pi|_{1,E}\vertiii{v_h}_{E} \\ 
+ \rho^{-1}b_{1}\|\widetilde{p}_I-\widetilde{p}_\pi\|_{0,\O}\|v_{h}\|_{0,\O} + \rho^{-1}\|\widetilde{p}-\widetilde{p}_\pi\|_{0,\O}\|v_{h}\|_{0,\O} \\
\leq C_{2}^{*}\left(\vertiii{\widetilde{p}_I-\widetilde{p}_\pi} + |\widetilde{p}-\widetilde{p}_\pi|_{1,h} + \|\widetilde{p}_I - \widetilde{p}_\pi\|_{0,\O}\right. \\
\left. + \|\widetilde{p}-\widetilde{p}_\pi\|_{0,\O}\right)\left(\vertiii{v_h}^{2} + \|v_h\|_{0,\O}^{2}\right)^{1/2},
\end{multline*}
where $C_{2}^*$ is a positive constant given by $$\displaystyle C_{2}^* := \max_{E\in\CT_h}\left\lbrace c^{2}\rho^{-1}C_{2}(E),c^{2}\rho^{-1} \sqrt{C_{4}(E)},\rho^{-1}b_{1}(E),\rho^{-1}\right\rbrace.$$ 

On the other hand, using the definition of $\vertiii{\cdot}$ we obtain
\begin{multline*}
\vertiii{\widetilde{p}-\widetilde{p}_I}^{2} = \displaystyle{\sum_{E\in\CT_h} \left[a^{E}(\PiK(\widetilde{p}-\widetilde{p}_I),\PiK(\widetilde{p}-\widetilde{p}_I))\right.} \\ 
\left. + S^{E}(\widetilde{p}-\widetilde{p}_I - \overline{\widetilde{p}-\widetilde{p}_I}, \widetilde{p}-\widetilde{p}_I - \overline{\widetilde{p}-\widetilde{p}_I})\right] \\
= \displaystyle{\sum_{E\in\CT_h} c^{2}\rho^{-1}|\PiK(\widetilde{p}-\widetilde{p}_I)|_{1,E}^{2} + S^{E}(\widetilde{p}-\widetilde{p}_I, \widetilde{p}-\widetilde{p}_I)} \\
= \displaystyle{\sum_{E\in\CT_h} c^{2}\rho^{-1}|\PiK(\widetilde{p}-\widetilde{p}_I)|_{1,E}^{2} + h_E|\widetilde{p}-\widetilde{p}_I|_{1,\partial E}^{2}} \\
\leq  \displaystyle{\sum_{E\in\CT_h} c^{2}\rho^{-1}|\widetilde{p}-\widetilde{p}_I|_{1,E}^{2} + h_E|\widetilde{p}-\widetilde{p}_I|_{1,\partial E}^{2}} \leq \max\{c^{2}\rho^{-1},1\}\displaystyle{\sum_{E\in\CT_h} h_E^{2s}|\widetilde{p}|_{1+s,E}^{2}} \\ 
\leq \max\{c^{2}\rho^{-1},1\}h^{2s}|\widetilde{p}|_{1+s,\O}^{2},
\end{multline*}
where in the last inequality we have used a scaled trace inequality. Hence, we obtain
\begin{equation*}
\vertiii{\widetilde{p}-\widetilde{p}_I} \leq \max\{c\rho^{-1/2},1\}h^{s}|\widetilde{p}|_{1+s,\O}.
\end{equation*}
The same arguments can be used for $\vertiii{\widetilde{p}-\widetilde{p}_\pi}$, obtaining 
\begin{equation*} 
\vertiii{\widetilde{p}-\widetilde{p}_\pi} \leq \max\{c\rho^{-1/2},1\}h^{s}|\widetilde{p}|_{1+s,\O}.
\end{equation*}
On the other hand, using Lemma \ref{eq:polyapprox} we obtain $|\widetilde{p}-\widetilde{p}_\pi|_{1,h} \leq Ch^{s}|\widetilde{p}|_{1+s,\O}$.
Then, using triangle inequality and the previous estimates, we conclude for \textrm{(I)} that 
\begin{equation*}
\textrm{(I)} \leq C_{3}^{*}h^{s}|\widetilde{p}|_{1+s,\O}\left(\vertiii{v_h}^{2} + \|v_{h}\|_{0,\O}^{2}\right)^{1/2}, \quad C_{3}^{*} := \max\{c\rho^{-1/2},1\}C_{2}^{*}.
\end{equation*}

Now, for \textrm{(II)} we have
\begin{multline*}
|\textrm{(II)}| \leq \sum_{E\in\CT_h} |b^{E}(P_hf,v_h) - b_h^E(P_hf,v_h)| \\ 
= \sum_{E\in\CT_h} |b^{E}(P_hf - f_\pi,v_h) - b_h^E(P_hf - f_\pi,v_h)| \\
\leq \rho^{-1}\max_{E\in\CT_h}\{b_{1}(E),1\}\|P_hf - f_\pi\|_{0,\O}\|v_{h}\|_{0,\O} \\
\leq \rho^{-1}\max_{E\in\CT_h}\{b_{1}(E),1\}(\|f-f_{I}\|_{0,\O} + \|f-f_\pi\|_{0,\O})\left(\vertiii{v_h}^{2} + \|v_{h}\|_{0,\O}^{2}\right)^{1/2} \\
\leq \rho^{-1}\max_{E\in\CT_h}\{b_{1}(E),1\}h^{s}\|f\|_{1,\O}\left(\vertiii{v_h}^{2} + \|v_{h}\|_{0,\O}^{2}\right)^{1/2},
\end{multline*}
where we have used Lemmas \ref{eq:interpolant}, \ref{eq:polyapprox}, and the best approximation property for $P_h$. Therefore, using Lemma \ref{lmm:regularity} we obtain
\begin{equation}\label{eq:triplel2}
\vertiii{v_h}^{2} + \|v_h\|_{0,\O}^{2} \leq C_4^{*}h^{s}\|f\|_{1,\O}\left(\vertiii{v_h}^{2} + \|v_h\|_{0,\O}^{2}\right)^{1/2},
\end{equation}
where $\displaystyle C_{4}^{*} := \max_{E\in\CT_h}\{C_{3}^{*},\rho^{-1}b_{1}(E),\rho^{-1}\}$ and hence, combining the previous estimate with \eqref{eq:stabbb}, we obtain $\|v_h\|_{1,\O} \leq C_4^*C_1^{*-1/2}h^{s}\|f\|_{1,\O}$.
Finally, defining the constant  $\mathcal{C} := \max\{C_4^*C_1^{*-1/2},1\}$, we conclude the proof.
\end{proof}

\begin{remark}
Let us remark that  \eqref{eq:triplel2} will be useful to derive the double order of convergence for eigenvalues.
\end{remark}

From the previous Lemma, we can conclude the convergence in norm for $\widehat{T}_{h}$ to $T$ as $h \to 0$. This is a key ingredient in order to obtain error estimates for eigenvalues and eigenfunctions.

We present as a consequence of the above, that the proposed method does not introduce spurious eigenvalues. In practical terms, this implies that isolated parts of $\sp(T)$ are approximated by isolated parts of $\sp(\widehat{T}_h)$ (see \cite{MR0203473}). This is contained in the following result.

\begin{theorem}
Let $G \subset \mathbb{C}$  be an open set containing $\sp(T)$. Then, there exists $h_0>0$ such that $\sp(\widehat{T}_h)\subset G$ for all $h<h_0$.
\end{theorem}

Let us remark that that the spectra of $T_{h}$ and $\widehat{T}_{h}$ coincide. According to this, let $\mu$ be an isolated eigenvalue if $T$ with multiplicity $m$ and let $\mathcal{E}$ be its associated eigenspace. Then, there exist $m$ eigenvalues $\mu_h^{(1)},\ldots, \mu_h^{(m)}$ of $T_h$, repeated according to their respective multiplicities that converge to $\mu$. Now, let $\mathcal{E}_h$ be the  direct sum of the associated eigenspaces of $\mu_h^{(1)},\ldots, \mu_h^{(m)}$. With these definitions at hand, now we focus on the analysis of error estimates.

 \subsection{Error estimates}
Now our task is to obtain error estimates for the approximation of the eigenvalues and eigenfunctions. With this goal in mind, first  we need to recall the definition of the  gap $\widehat{\delta}$ between two closed subspaces $\CM$ and $\CN$ of  $\H^1(\Omega)$:
\begin{equation*}
\widehat{\delta}(\CM,\CN) := \max\{\delta (\CM, \CN), \delta (\CM, \CN)\},
\end{equation*}
where $\delta (\CM,\CN):= \underset{x \in \CM : \|x\|_{1, \Omega} = 1}{\sup} \left\lbrace \underset{y \in \CN}{\inf} \quad \|x-y\|_{1, \Omega} \right\rbrace$.

The following result provides an error estimate for the eigenfunctions and eigenvalues.
\begin{theorem}
\label{thm:errors1} 
The following estimates hold
\begin{equation*}
\widehat{\delta}(\mathcal{E}, \mathcal{E}_{h}) \leq \mathcal{C}h^{r}, \quad |\mu - \mu_{h}^{(i)}| \leq \mathcal{C}h^{r}, \quad i = 1, \ldots, m,
\end{equation*}
where $r$ is given in Lemma \ref{lmm:regularity}.
\end{theorem}

\begin{proof}
Thanks to Lemma \ref{lmm:conv_norm},  the proof is a direct consequence of the compact operators theory  of  Babu\v{s}ka-Osborn (see \cite[Theorems 7.1 and 7.3]{BO}).
\end{proof}

Let us observe that Theorem \ref{thm:errors1} is a result with a preliminary error estimate for the eigenvalues. However, it is possible to improve the order of convergence for the eigenvalues as we prove on the following result.
\begin{theorem}
The following estimate holds
	\begin{equation*}
	\label{eq:double_order}
		|\lambda-\lambda_h^{(i)}|\leq \mathcal{K} h^{2r},
	\end{equation*}
where $\mathcal{K}>0$ is a constant independent of $h$ and  $r$ is given in Lemma \ref{lmm:regularity}.
\end{theorem}
\begin{proof}
Let $(\l_h^{(i)},p_h)\in\mathbb{R}\times V_h$ be   a solution of \eqref{eq:spectral_disc} with $\left\|p_h\right\|_{1,\O}=1$. According to Theorem \ref{lmm:conv_norm}, there exists a solution $(\l,p)\in\mathbb{R}\times \H^1(\O)$ of the eigenvalue problem \eqref{eq:pression} such that $\left\|p-p_h\right\|_{1,\O}\leq Ch^{r}$. 

From the symmetry of the bilinear forms and the facts that $\widehat{a}(p,v)=\l b(p,v)$ for all $v\in\H ^1(\O)$ (cf.\eqref{eq:pression}) and $\widehat{a}_h(p_h,v_h)=\l_h^{(i)}b(p_h,v_h)$ for all $v_h\in\Vh$ (cf.\eqref{eq:spectral_disc}), we have
\begin{multline*}
\widehat{a}(p-p_h,p-p_h)-\l b(p-p_h,p-p_h)
 =\widehat{a}(p_h,p_h)-\l b(p_h,p_h)
\\
 =\left[\widehat{a}(p_h,p_h)-\widehat{a}_h(p_h,p_h)\right]
-\left(\l-\l_h^{(i)}\right)b(p_h,p_h),
\end{multline*}
from which we obtain the following identity:
\begin{multline*}
\label{eq:padra}
(\l_h^{(i)}-\l)b(p_h,p_h)
=\underbrace{\widehat{a}(p-p_h,p-p_h)-\l b(p-p_h,p-p_h)}_{T_1}\\
+\underbrace{\left[\widehat{a}_h(p_h,p_h)-\widehat{a}(p_h,p_h)\right]}_{T_2}+\underbrace{\lambda_h^{(i)}[b(p_h,p_h)-b_h(p_h,p_h)]}_{T_3}.
\end{multline*}
Now our task is to estimate the contributions $T_1$, $T_2$ and $T_3$ of the right-hand side. For the term $T_1$ we invoking Lemma \ref{lmm:conv_norm} in order to obtain
\begin{multline}
\label{eq:T1}
\displaystyle |T_1|\leq \mathcal{C}c^{2}\rho^{-1} \|p-p_h\|_{1,\O}^2+\|p-p_h\|_{0,\O}^2\\
\leq \max\{c^{2}\rho^{-1},1\}\|p-p_h\|_{1,\O}^{2}
\leq \mathcal{C}^{2}\max\{c^{2}\rho^{-1},1\}h^{2r}.
\end{multline}
The estimate for $T_2$ is obtained as follows
\begin{multline*}
|T_2|\leq \sum_{E\in\CT_h}| \widehat{a}_h^E(p_h,p_h)-\widehat{a}^E(p_h,p_h)|\\
= \sum_{E\in\CT_h} |\widehat{a}_h^E(p_h - p_\pi,p_h - p_\pi) - \widehat{a}^E(p_h - p_\pi,p_h - p_\pi)| \\
\leq \sum_{E\in\CT_h} c^{2}\rho^{-1}C_2(E)\vertiii{p_h - p_\pi}_E^2 + c^{2}\rho^{-1}|p_h - p_\pi|_{1,E}^2 + \rho^{-1}\max\{b_{1}^{E},1\}\|p_h-p_\pi\|_{0,\O}^{2} \\
\leq C_5^{*}\sum_{E\in\CT_h} \left(\vertiii{p - p_I}_E^2 +  \vertiii{p_h - p_I}_E^2 + \vertiii{p - p_\pi}_E^2 + |p - p_\pi|_{1,E}^2 \right. \\ 
\left.  + |p - p_h|_{1,E}^2 + \|p_{h} - p_{I}\|_{0,\O}^{2} + \|p - p_I\|_{0,\O}^{2} + \|p-p_\pi\|_{0,\O}^{2}\right),
\end{multline*}
where $\displaystyle C_{5}^{*} := \max_{E\in\CT_h}\{c^{2}\rho^{-1}C_2(E),c^{2}\rho^{-1},\rho^{-1}b_{1}(E),\rho^{-1}\}$. 
Now, invoking \eqref{eq:triplel2} and Lemmas \ref{eq:polyapprox}, \ref{eq:interpolant} and \ref{lmm:conv_norm}, we obtain
\begin{equation}
\label{eq:T2}
|T_2| \leq C_6^{*}h^{2r}, \quad C_{6}^{*} := C_{5}^{*}\max\{C_3^{*2}, \mathcal{C}^{2}, \rho^{-1},1\}.
\end{equation}
On the other hand, using approximation properties for $\Pi$, we obtain for $T_3$
\begin{equation}
\label{eq:T3}
|T_3| \leq \max_{E\in\CT_h}\{b_{1}(E)\}\|p_h-\Pi p_h\|_{0,\O}^2 \leq \max_{E\in\CT_h}\{b_{1}(E)\}\mathcal{C}h^{2r}.
\end{equation}
Finally, using the fact that $\lambda_{h}^{(i)} \rightarrow 0$ as $h\to0$, we obtain
\begin{equation*}
b_{h}(p_h,p_h) = \dfrac{\widehat{a}_h(p_h,p_h)}{\l_h^{(i)}} \geq \dfrac{\underline{C}\|p_h\|_{1,\O}^{2}}{\l_{h}^{(i)}} \geq \widetilde{C} > 0.
\end{equation*}
Therefore, gathering \eqref{eq:T1}, \eqref{eq:T2} and \eqref{eq:T3} and defining 
\begin{equation*}
\mathcal{K}:= \max\left\lbrace \mathcal{C}^{2}\max\{c^{2}\rho^{-1},1\}, C_{6}^{*}, \max_{E\in\CT_h}\{b_{1}(E)\}\mathcal{C} \right\rbrace,
\end{equation*}
we conclude the proof.
\end{proof}
\subsection{Error estimates in $\L^2$ norm}
In the present subsection we establish error estimates for eigenfunctions in $L^{2}$ norm. We begin this subsection with the following result, where a classical duality argument has been used.
\begin{lemma}\label{eq:cvlm2}
Let $f \in \mathcal{E}$ be such that $\widetilde{p} := Tf$ and $\widetilde{p}_{h} := \widehat{T}_{h}f$. Then, the following estimate  holds
\begin{equation*}
\|\widetilde{p}-\widetilde{p}_{h}\|_{0,\O} \leq \mathcal{J}h^{\widetilde{r}+s}\|f\|_{1,\O},
\end{equation*}
where $\mathcal{J}$ is a positive constant independent of $h$, and $\widetilde{r}$ is given by Lemma \ref{lmm:regularity}.
\end{lemma}

\begin{proof}
Let us consider the following auxiliarly problem: Find $q \in \H^1(\O)$ such that 
\begin{equation}\label{eq:auxprob}
\widehat{a}(q,v) = b(\widetilde{p}-\widetilde{p}_{h},v) \quad \forall v \in \H^1(\O).
\end{equation}
Observe that this problem is well-posed and there exists $\widetilde{r}$ as in Lemma \ref{lmm:regularity}, such that 
\begin{equation*}
|q|_{1+\widetilde{r},\O} \leq C\|\widetilde{p}-\widetilde{p}_{h}\|_{0,\O}.
\end{equation*}
Now, testing \eqref{eq:auxprob} with $v := \widetilde{p}-\widetilde{p}_{h}$, we obtain the following identity
\begin{equation*}
\begin{split}
\|\widetilde{p}-\widetilde{p}_{h}\|_{0,\O}^{2} &= b(\widetilde{p}-\widetilde{p}_{h},\widetilde{p}-\widetilde{p}_{h}) = \widehat{a}(q,\widetilde{p}-\widetilde{p}_{h}) \\
&= \widehat{a}(q-q_{I},\widetilde{p}-\widetilde{p}_{h}) + \widehat{a}(q_{I},\widetilde{p}-\widetilde{p}_{h}) \\
&= \underbrace{\widehat{a}(q-q_{I},\widetilde{p}-\widetilde{p}_{h})}_{B_{1}} + \underbrace{b(f,q_{I}) - b_{h}(P_hf,q_{I})}_{B_{2}} + \underbrace{\widehat{a}_{h}(q_{I},\widetilde{p}_{h}) - \widehat{a}(q_{I},\widetilde{p}_{h})}_{B_{3}}.
\end{split}
\end{equation*}
Then, we need to estimate the contributions on the right-hand side of the above equation. To estimate $B_{1}$, we invoke Lemma \ref{lmm:conv_norm} and Lemma \ref{eq:interpolant} in order to obtain
\begin{equation}\label{eq:B1}
|B_{1}| \leq c^{2}\rho^{-1}|\widetilde{p}-\widetilde{p}_{h}|_{1,\O}|q-q_{I}|_{1,\O} \leq  c^{2}\rho^{-1}\mathcal{C}h^{\widetilde{r}+s}\|f\|_{1,\O}|q|_{1+\widetilde{r},\O}.
\end{equation}
For $B_{2}$, using error estimates for $\Pi$ and triangle inequality, we obtain
\begin{multline}\label{eq:B2}
|B_{2}| = |b(P_hf,q_I) - b_h(P_hf,q_I)| 
\leq \max_{E\in\CT_h}\{b_{1}(E),1\}\rho^{-1}\|f - \Pi f\|_{0,\O}\|q_{I} - \Pi q_{I}\|_{0,\O} \\
\leq \max_{E\in\CT_h}\{b_{1}(E),1\}\rho^{-1}h^{s}\|f\|_{1,\O}\left(\|q-q_{I}\|_{0,\O} + \|q-\Pi q\|_{0,\O} + \|\Pi(q-q_{I})\|_{0,\O}\right) \\
\leq \max_{E\in\CT_h}\{b_{1}(E),1\}\rho^{-1}h^{\widetilde{r}+s}\|f\|_{1,\O}|q|_{1+\widetilde{r},\O}.
\end{multline}
Finally, for $B_{3}$ we invoke \eqref{eqrefgt1} and \eqref{eqrefgt3} in order to obtain 
\begin{multline}\label{eq:B3}
|B_{3}| \leq \sum_{E\in\CT_h} |\widehat{a}_{h}^{E}(\widetilde{p}_{h},q_I) - \widehat{a}^{E}(\widetilde{p}_{h},q_I)| \\ 
= \sum_{E\in\CT_h} |\widehat{a}_{h}^{E}(\widetilde{p}_{h} - p_\pi,q_I - q_\pi) - \widehat{a}^{E}(\widetilde{p}_{h} - p_\pi,q_I - q_\pi)| \\
\leq C_{7}^{*}\sum_{E\in\CT_h} \vertiii{q_I - q_\pi}_{E}\vertiii{\widetilde{p}_{h}-p_\pi}_{E}  + \|\widetilde{p}_{h}-p_{\pi}\|_{0,\O}\|q_\pi-q_{I}\|_{0,\O} \\
\leq C_{7}^{*}\left[\left(\vertiii{q-q_I} + \vertiii{q-q_\pi}\right)\left(\vertiii{\widetilde{p}-\widetilde{p}_{I}} + \vertiii{\widetilde{p}_{h} - \widetilde{p}_{I}} + \vertiii{\widetilde{p}-p_\pi}\right)\right. \\
\left. + \left(\|\widetilde{p}-\widetilde{p}_{I}\|_{0,\O} + \|\widetilde{p}_I - \widetilde{p}_{h}\|_{0,\O} + \|\widetilde{p} - p_\pi\|_{0,\O}\right)\left(\|q-q_I\|_{0,\O} + \|q-q_\pi\|_{0,\O}\right)\right].
\end{multline}
where $\displaystyle C_{7}^{*} := c^{2}\rho^{-1}\max_{E\in\CT_h}\{C_{2}(E),C_{4}(E), b_{1}(E),1\}$.
Hence, using \eqref{eq:triplel2}, we conclude that $|B_{3}| \leq \widehat{C}h^{\widetilde{r}+s}\|f\|_{1,\O}|q|_{1+\widetilde{r},\O}$,
where the constant $\widehat{C}$ is defined by
$$\widehat{C} := \max\{C_{7}\sqrt{C_{4}^{*}}\max\{c^{2}\rho^{-1},1\}, \max\{c^{2}\rho^{-1},1\}\}.$$
Therefore, combining \eqref{eq:B1}, \eqref{eq:B2} and \eqref{eq:B3}, together with the additional regularity for $q$, and setting $$\mathcal{J}:= \max\{c^{2}\rho^{-1}\mathcal{C},\max_{E\in\CT_h}\{b_{1}(E),1\}, \widehat{C}\},$$ we  obtain
\begin{equation*}
\|(T-\widehat{T}_h)f\|_{0,\O} = \|\widetilde{p}-\widetilde{p}_{h}\|_{0,\O} \leq \mathcal{J}h^{\widetilde{r}+s}\|f\|_{1,\O},
\end{equation*}
concluding the proof.
\end{proof}

Now, we introduce the solution operator on the space $\L^{2}(\O)$, given by
$$\widetilde{T} : \L^{2}(\O) \rightarrow \L^{2}(\O), \quad \widetilde{f} \mapsto \widetilde{T}\widetilde{f} := \widetilde{p},$$
where $\widetilde{p}$ is the (unique) solution of the associated source problem. It is easy to check that $\widetilde{T}$ is compact and self-adjoint with respect to $\widehat{a}(\cdot,\cdot)$. Moreover, the spectra of $T$ and $\widetilde{T}$ coincide. 

Now, we are in position to prove the convergence in the $\L^{2}$ norm for $\widehat{T}_{h}$ to $\widetilde{T}$ as $h \to 0$.
\begin{lemma}\label{eq:cvvnormL2}
The following estimate holds
\begin{equation*}
\|(\widetilde{T}-\widehat{T}_{h})\widetilde{f}\|_{0,\O} \leq \widetilde{\mathcal{J}}h^{s}\|\widetilde{f}\|_{0,\O} \quad \forall \widetilde{f} \in \L^{2}(\O),
\end{equation*}
where $s$ is given in Lemma \ref{lmm:regularity}.
\end{lemma}

\begin{proof}
The proof follows the same arguments that those in the proof of Lemma \ref{lmm:conv_norm}, but the term $\mathrm{(II)}$ must be estimated by follows:
\begin{multline*}
|\mathrm{(II)}| \leq \sum_{E\in\CT_h} |b^{E}(P_hf,v_h) - b_h^E(P_hf,v_h)| \\
= \sum_{E\in\CT_h} |b^E(P_hf,v_h - (v_h)_\pi) - b_h^E(P_hf,v_h - (v_h)_\pi)| \\ 
\leq \max_{E\in\CT_h}\{b_{1}(E),1\}\rho^{-1}\|P_hf\|_{0,\O}\|v_h - (v_h)_\pi\|_{0,\O} \\
 \leq \max_{E\in\CT_h}\{b_{1}(E),1\}\rho^{-1}h^{s}\|f\|_{0,\O}\|v_h\|_{1,\O},
\end{multline*}
where, invoking \eqref{eqrefgt3} we obtain $$\displaystyle \|v_h\|_{1,\O} \leq c\rho^{-1/2}\max_{E\in\CT_h}\{C_{4}(E),1\}\left(\vertiii{v_h}^{2} + \|v_h\|_{0,\O}^{2}\right)^{1/2}.$$ This concludes the proof with
\begin{equation*}
\displaystyle \widetilde{\mathcal{J}} := \max\left\lbrace c\rho^{-1/2}\max_{E\in\CT_h}\{C_{4}(E),1\},C_3^*,\rho^{-1}b_{1}(E),\rho^{-1} \right\rbrace.
\end{equation*}
\end{proof}

As a consequence of the previous Lemma, a spectral convergence result analogous to Theorem \ref{thm:errors1} holds for $\widetilde{T}$ and $\widehat{T}_{h}$. This allows us to obtain the following result.

\begin{theorem}
Let $p_{h}$ be an eigenfunction of $\widehat{T}_{h}$ associated to the eigenvalue $\mu_{h}^{(i)}$, $1 \leq i \leq m$ with $\|p_{h}\|_{0,\O} = 1$. Then, there exists an eigenfunction $p$ of $\widetilde{T}$ associated to the eigenvalue $\mu$ such that 
\begin{equation*}
\|p-p_{h}\|_{0,\O} \leq \mathcal{J}h^{\widetilde{r}+s}\|f\|_{1,\O},
\end{equation*}
where $\mathcal{J}$ is a positive constant independent of $h$.
\end{theorem}

\begin{proof}
Observe that invoking Lemma \ref{eq:cvvnormL2} and \cite[Theorem 7.1]{BO}, we have spectral convergence of $\widehat{T}_{h}$ to $\widetilde{T}$. On the other hand, due to the relation between the eigenfunctions of $T$ and $T_{h}$ with those of $\widetilde{T}$ and $\widehat{T}_{h}$ respectively, we have $p_{h} \in \widetilde{\mathcal{E}}_{h}$ and there exists $p \in \mathcal{E}$ such that
\begin{equation}\label{eq12345}
\|p-p_{h}\|_{0,\O} \leq C\underset{\widetilde{f}\in\widetilde{\mathcal{E}} : \|\widetilde{f}\|_{0,\O}=1}{\sup} \|(\widetilde{T}-\widehat{T}_{h})\widetilde{f}\|_{0,\O}.
\end{equation}
Then, invoking Lemma \ref{eq:cvlm2}, for every $\widetilde{f} \in \widetilde{\mathcal{E}}$, if $f \in \mathcal{E}$ is such that $\widetilde{f} = f$ then
\begin{equation*}
\|(\widetilde{T}-\widehat{T}_{h})\widetilde{f}\|_{0,\O} = \|(T-\widehat{T}_{h})f\|_{0,\O} \leq \widetilde{\mathcal{J}}h^{\widetilde{r}+s}\|f\|_{1,\O}.
\end{equation*}
Finally, since $f \in \mathcal{E}$, we have that $\|f\|_{1,\O} = \mu^{-1}\|Tf\|_{1,\O} \leq C\|f\|_{0,\O}$, and combining it with \eqref{eq12345}, we conclude the proof.
\end{proof}


\section{Numerical experiments}
\label{sec:numerics}
Now we present a number of numerical tests to illustrate the performance of the proposed method. All the results have been obtained 
with a Matlab code. Since we are interested on the versatility of the method, we focus our tests for two type of domains: convex and non-convex. It is well known that for convex domains the eigenfunctions are sufficiently smooth compared with the ones associated to non-convex domains, which is reflected on the convergence order for the eigenvalues This must be captured with our method. Hence, the convergence orders and extrapolated values for the frequencies are obtained by means of a standard least-square fitting of the form 
\begin{equation}
\label{eq:least_square}
\omega_{hi} \approx \omega_{i} + C_{i}h^{\xi_{i}},
\end{equation}
where $\xi_{i}$ represents the computed order of convergence. Through all this section,  $\mathrm{N}$  represents the mesh refinement which is considered as the number of polygons on the bottom of the domain. 

For the construction of meshes allowing for small edges we recall the following procedure introduced in \cite{danilo_eigen}: 
\begin{enumerate}
\item[\textbf{Step 1}]  For any polygon $E\in\CT_h$,  we add  a hanging node on each edge of every polygon. This hanging node, denoted by $x_{HG}$, is constructed by using the following convex combination:
\begin{equation*}
x_{HG} := (1-t)x_{V}^{1} + tx_{V}^{2}, \quad t := \dfrac{\text{dist}(x_{V}^{1},x_{V}^{2})}{M},
\end{equation*}
where $x_{V}^{1}$ and $x_{V}^{2}$ are the vertices of the corresponding edge for which the hanging node has been added and $M > 0$.\\
\item[\textbf{Step 2}] The hanging node on \textbf{Step 1} now  is displaced along the edge with respect to a vertex of the polygon using the parameter $M > 0$, in order to make this hanging node collapse with such vertex in each refinement. More precisely, if $x_{HG}$ is close to $x_{V}^{1}$, the distance between $x_{V}^{1}$ and $x_{HG}$ is 
\begin{equation*}
\text{dist}(x_{HG},x_{V}^{1}) = \dfrac{\text{dist}(x_{V}^{1},x_{V}^{2})^{2}}{M}.
\end{equation*}
\end{enumerate}
Hence, the idea is taking $M > 0$ increasing in each mesh refinement, in order to obtain a hanging node that collapse with one of the vertices, and therefore, small edges. 

First, we define the ratio as $\text{Ratio} = h_{m}(E)/h_{E}$, where $h_{E}$ is the diameter of $E$ and $h_{m}(E)$ is the shortest edge of $E$. In Figure \ref{fig:meshesx} we present some polygonal meshes considered for the numerical experiments.
\begin{figure}[H]
	\begin{center}
			\centering\includegraphics[height=4.4cm, width=4.2cm]{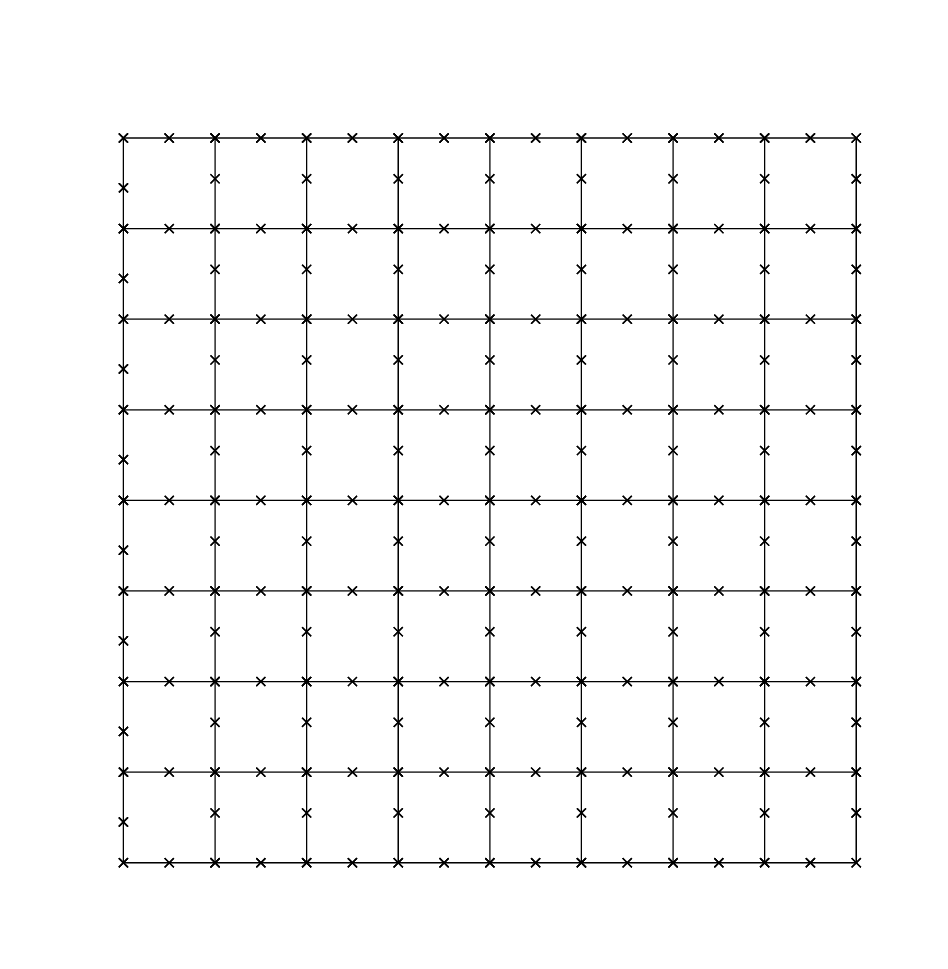}
			\centering\includegraphics[height=4.4cm, width=4.2cm]{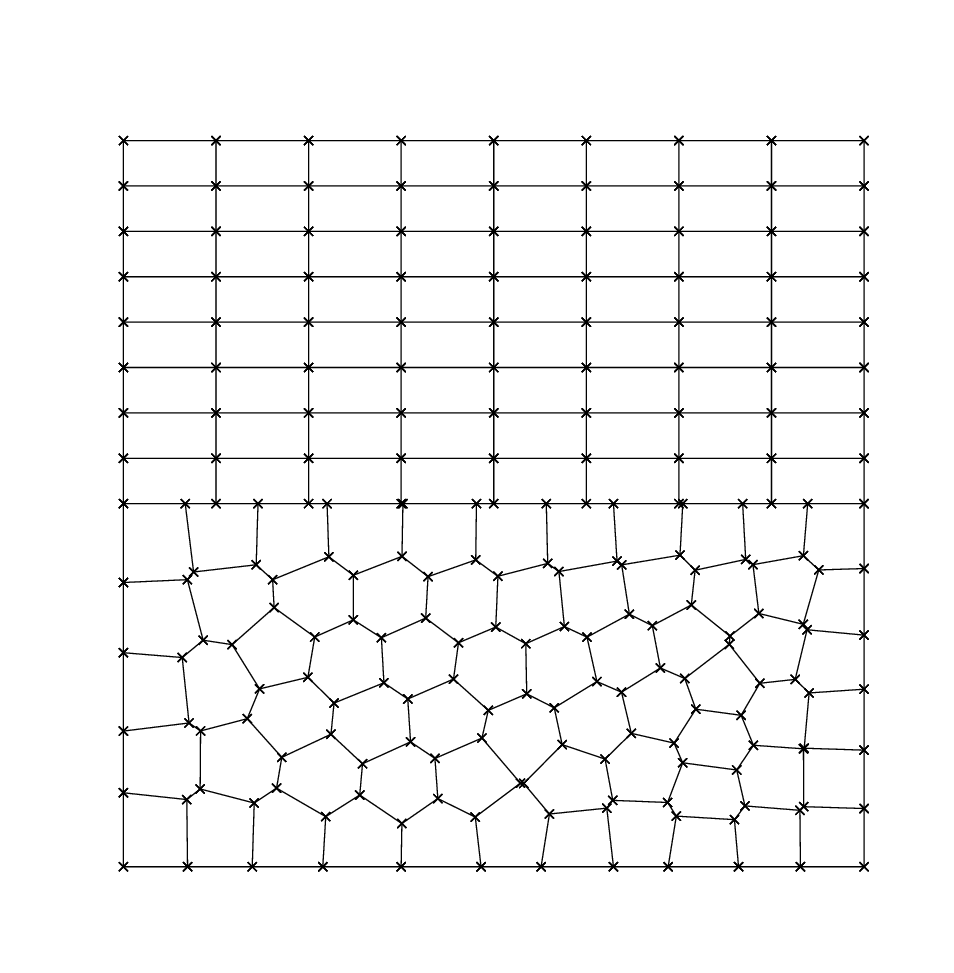}
			\centering\includegraphics[height=4.4cm, width=4.2cm]{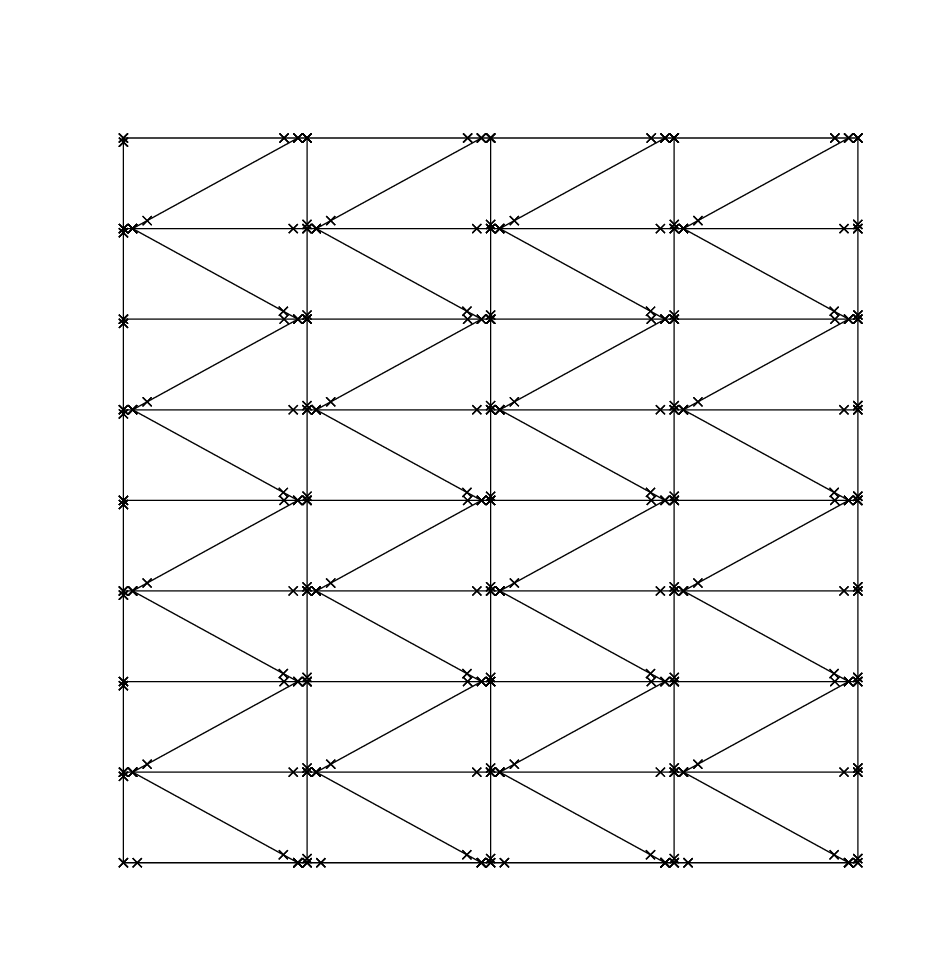}
		\caption{Sample of meshes. Left: $\mathcal{T}_{h}^{1}$ ($\mathrm{N}$ = 8); center: $\mathcal{T}_{h}^{2}$ ($\mathrm{N}$ = 11); right: $\mathcal{T}_{h}^{3}$ ($\mathrm{N}$ = 8).
}
		\label{fig:meshesx}
	\end{center}
\end{figure}

Let us mention that the meshes $\CT_{h}^{1}$ and $\CT_h^{3}$ was constructed in the process described previously, whereas mesh $\CT_h^{2}$ has not hanging nodes. However, this mesh is constructed with small edges.

\subsection{Test 1: Rectangular  acoustic cavity}
For this test, the computational domain is a rectangle of the form $\O:=(0,a)\times (0,b)$. For this domain, if we consider the physical parameters equal to one ($\rho=c=1$), the exact eigenvalues and eigenfunctions are know and are of the form 
\begin{align*}
\lambda_{nm}&:=\displaystyle\pi^2\left(\left(\frac{n}{a} \right)^2+\left(\frac{m}{b} \right)^2 \right),\quad n,m=0,1,2,\ldots, n+m\neq 0\\
\\
\bu_{nm}(x,y)&:=\begin{pmatrix}\displaystyle \frac{n}{a}\sin\left(\frac{n\pi x}{a}\right)\cos\left(\frac{m\pi y}{b}\right)\\
\\
\displaystyle \frac{m}{b}\cos\left(\frac{n\pi x}{a}\right)\sin\left(\frac{m\pi y}{b}\right)\end{pmatrix},
\end{align*}
where $\bu_{nm}(x,y)$ corresponds to the displacement of the fluid which can be computed by the relation $\displaystyle\frac{\nabla p_{nm}}{\lambda_{nm}}=\bu_{nm}$ that holds for the acoustic problem. Moreover we have considered the theoretical stabilization term, that is 
\begin{equation}\label{stabspec}
S(p_{h},q_{h}) = \displaystyle{\sum_{E \in \mathcal{T}_{h}} S^{E}(p_{h},q_{h})}, \quad S^{E}(p_{h},q_{h}) = \sigma_{E}h_{E}\displaystyle{\int_{E} \partial_{s}p_{h}\partial_{s}q_{h}},
\end{equation}
where in this case, the stabilization parameter has been taken as $\sigma_{E} = 1$, whereas the constanst $a$ and $b$ has been taken as $a=1$ and $b=1.1$. 

In Tables \ref{tabla1}, \ref{tabla2} and \ref{tabla3} we report approximated values of each one of the frequencies $\omega_{i} = \sqrt{\lambda_{i}-1}$, $i = 1,\ldots,5$, their respective orders of convergence and extrapolated frequencies for different meshes presented in Figure \ref{fig:meshesx}.
Also, in each table, in the row "Ratio" we report the measure of the on each refinement, in order to present the presence of real small edges on the meshes.
\begin{table}[H]
\caption{Five lowest approximated frequencies, orders of convergence, extrapolated frequencies and ratios, computed with $\mathcal{T}_{h}^{1}$ and the stabilization term defined in \eqref{stabspec}.}
\label{tabla1}
\begin{center}
\resizebox{13cm}{!}{
\begin{tabular}{|c|c|c|c|c|c|c|c|} \hline
$\omega_{hi}$ & N = 32 & N = 64 & N = 128 & N = 256 & Order & Exact. \\ \hline 
 $\omega_{h1}$ & 0.82710 & 0.82661 & 0.82649 & 0.82646 & 1.99 & 0.82645 \\
$\omega_{h2}$ & 1.00079 & 1.00020 & 1.00005 & 1.00001 & 1.98 & 1.00000 \\
$\omega_{h3}$ & 1.83520 & 1.82864 & 1.82699 & 1.82658 & 2.00 & 1.82645 \\
$\omega_{h4}$ & 3.31640 & 3.30844 & 3.30645 & 3.30595 & 2.00 & 3.30579 \\
$\omega_{h5}$ & 4.01300 & 4.00321 & 4.00080 & 4.00020 & 2.02 & 4.00001 \\
\hline
Ratio & 7.1638e-03 &  1.7794e-03 &  4.4413e-04 &  1.1099e-04 \\ \cline{1-5}
\end{tabular}}
\end{center}

\end{table}

\begin{table}[H]
\caption{Five lowest approximated frequencies, orders of convergence, extrapolated frequencies and ratios, computed with $\mathcal{T}_{h}^{2}$ and the stabilization term defined in \eqref{stabspec}.}
\label{tabla2}
\begin{center}
\resizebox{13cm}{!}{
\begin{tabular}{|c|c|c|c|c|c|c|c|} \hline
$\omega_{hi}$ & N = 11 & N = 20 & N = 39 & N = 88 & Order & Exact. \\ \hline 
 $\omega_{h1}$ & 0.82796 & 0.82685 & 0.82655 & 0.82647 & 1.89 & 0.82645 \\
$\omega_{h2}$ & 1.00287 & 1.00074 & 1.00018 & 1.00005 & 1.97 & 1.00002 \\
$\omega_{h3}$ & 1.84173 & 1.83016 & 1.82737 & 1.82668 & 2.07 & 1.82658 \\
$\omega_{h4}$ & 3.32838 & 3.31203 & 3.30739 & 3.30619 & 1.86 & 3.30587 \\
$\omega_{h5}$ & 4.04570 & 4.01182 & 4.00294 & 4.00074 & 1.97 & 4.00028 \\
\hline
Ratio & 5.5513e-03 & 5.0032e-03 & 1.1050e-03 & 3.3969e-04 \\ \cline{1-5}
\end{tabular}}
\end{center}
\end{table}

\begin{table}[H]
\caption{Five lowest approximated frequencies, orders of convergence, extrapolated frequencies and ratios, computed with $\mathcal{T}_{h}^{3}$ and the stabilization term defined in \eqref{stabspec}.}
\label{tabla3}
\begin{center}
\resizebox{13cm}{!}{
\begin{tabular}{|c|c|c|c|c|c|c|c|} \hline
$\omega_{hi}$ & N = 32 & N = 64 & N = 128 & N = 256 & Order & Exact. \\ \hline 
 $\omega_{h1}$ & 0.82709 & 0.82661 & 0.82649 & 0.82646 & 1.97 & 0.82645 \\
$\omega_{h2}$ & 1.00292 & 1.00074 & 1.00018 & 1.00005 & 1.97 & 1.00000 \\
$\omega_{h3}$ & 1.83576 & 1.82884 & 1.82705 & 1.82660 & 1.96 & 1.82644 \\
$\omega_{h4}$ & 3.31611 & 3.30841 & 3.30644 & 3.30595 & 1.97 & 3.30577 \\
$\omega_{h5}$ & 4.04669 & 4.01188 & 4.00301 & 4.00075 & 1.97 & 3.99996 \\
\hline
Ratio & 3.7655e-07  & 9.2970e-08  & 2.3171e-08 &  5.7882e-09 \\ \cline{1-5}
\end{tabular}}
\end{center}
\end{table}

From Tables \ref{tabla1}, \ref{tabla2} and \ref{tabla3} we observe that our method is sharp on the approximation of the frequencies. We confirm this fact by comparing our results with the exact ones presented in each column "Exact". Moreover, the reported Ratios are smaller when the meshes are refined, confirming the presence of small edges on the meshes. These results are similar for each of the meshes. 

Let us remark the following: for the theoretical analysis we have used the stabilization \eqref{stabspec} that depend on the tangential derivatives. This is a theoretical 
argument for the small edges treatment of the theory. However, from the computational point of view, we are free to consider any stabilization for the method. This motivates the analysis of the robustness of the VEM with small edges using other stabilizations. With this goal in mind, let us compare the previous results with the following
where we implement  the following stabilization term
\begin{equation}\label{classicspec}
S(p_{h},q_{h}) = \displaystyle{\sum_{E \in \mathcal{T}_{h}} S^{E}(p_{h},q_{h})}, \quad S^{E}(p_{h},q_{h}) = \sigma_{E}\displaystyle{\sum_{i = 1}^{N_{E}} p_{h}(V_{i}) q_{h}(V_{i})},
\end{equation}
where $V_{i}$ represent each vertex of the polygon $E$, and $N_{E}$ represent the number of vertices of $E$.  Again, we consider $\sigma_{E}=1$ in order to compare the obtained results with those in the previous tests. In Tables \ref{tabla4}, \ref{tabla5} \ref{tabla6} we report approximated values of each one of the frequencies $\omega_{i} = \sqrt{\lambda_{i}-1}$, $i = 1,\ldots,5$, their respective orders of convergence, and extrapolated frequencies for the meshes presented in Figure \ref{fig:meshesx}.

\begin{table}[H]
\caption{Five lowest approximated frequencies, orders of convergence, extrapolated frequencies and ratios, computed with $\mathcal{T}_{h}^{1}$ and the stabilization term defined in \eqref{classicspec}.}
\label{tabla4}
\begin{center}
\resizebox{13cm}{!}{
\begin{tabular}{|c|c|c|c|c|c|c|c|} \hline
$\omega_{hi}$ & N = 32 & N = 64 & N = 128 & N = 256 & Order & Exact \\ \hline 
 $\omega_{h1}$ & 0.82710 & 0.82659 & 0.82649 & 0.82647 & 2.21 & 0.82645 \\
  $\omega_{h2}$ & 1.00078 & 1.00017 & 1.00005 & 1.00001 & 2.25 & 1.00001 \\
  $\omega_{h3}$ & 1.83527 & 1.82866 & 1.82700 & 1.82658 & 1.99 & 1.82644 \\
 $\omega_{h4}$ & 3.31656 & 3.30853 & 3.30645 & 3.30595 & 1.96 & 3.30575 \\
 $\omega_{h5}$ & 4.01356 & 4.00320 & 4.00080 & 4.00020 & 2.10 & 4.00005 \\
\hline
Ratio & 7.1638e-03 &  1.7794e-03 &  4.4413e-04 &  1.1099e-04 \\ \cline{1-5}
\end{tabular}}
\end{center}
\end{table}

\begin{table}[H]
\caption{Five lowest approximated frequencies, orders of convergence, extrapolated frequencies and ratios, computed with $\mathcal{T}_{h}^{2}$ and the stabilization term defined in \eqref{classicspec}.}
\label{tabla5}
\begin{center}
\resizebox{13cm}{!}{
\begin{tabular}{|c|c|c|c|c|c|c|c|} \hline
$\omega_{hi}$ & N = 11 & N = 20 & N = 39 & N = 88 & Order & Exact. \\ \hline 
 $\omega_{h1}$ & 0.82721 & 0.82663 & 0.82649 & 0.82646 & 2.09 & 0.82645 \\
 $\omega_{h2}$ & 1.00205 & 1.00051 & 1.00013 & 1.00003 & 2.03 & 1.00002 \\
 $\omega_{h3}$ & 1.83098 & 1.82757 & 1.82672 & 1.82651 & 2.04 & 1.82648 \\
 $\omega_{h4}$ & 3.31806 & 3.30869 & 3.30651 & 3.30596 & 2.11 & 3.30590 \\
 $\omega_{h5}$ & 4.03224 & 4.00816 & 4.00203 & 4.00050 & 2.00 & 4.00021 \\
\hline
Ratio & 5.5513e-03 & 5.0032e-03 & 1.1050e-03 & 3.3969e-04 \\ \cline{1-5}
\end{tabular}}
\end{center}
\end{table}

\begin{table}[H]
\caption{Five lowest approximated frequencies, orders of convergence, extrapolated frequencies and ratios, computed with $\mathcal{T}_{h}^{3}$ and the stabilization term defined in \eqref{classicspec}.}
\label{tabla6}
\begin{center}
\resizebox{13cm}{!}{
\begin{tabular}{|c|c|c|c|c|c|c|c|} \hline
$\omega_{hi}$ & N = 32 & N = 64 & N = 128 & N = 256 & Order & Exact. \\ \hline 
 $\omega_{h1}$ & 0.82695 & 0.82657 & 0.82647 & 0.82645 & 2.03 & 0.82645 \\
$\omega_{h2}$ & 1.00209 & 1.00054 & 1.00013 & 1.00003 & 1.96 & 1.00000 \\
$\omega_{h3}$ & 1.83278 & 1.82806 & 1.82688 & 1.82655 & 1.97 & 1.82645 \\
$\omega_{h4}$ & 3.31355 & 3.30771 & 3.30623 & 3.30591 & 2.02 & 3.30578 \\
$\omega_{h5}$ & 4.03362 & 4.00858 & 4.00201 & 4.00055 & 1.97 & 3.99992 \\
\hline
Ratio & 3.7655e-07  & 9.2970e-08  & 2.3171e-08 &  5.7882e-09 \\ \cline{1-5}
\end{tabular}}
\end{center}
\end{table}
Note that there are not significant  differences when the theoretical stabilization term \eqref{stabspec} is changed to \eqref{classicspec}. The results presented in Tables \ref{tabla4}, \ref{tabla5} and \ref{tabla6} are very similar to those presented in Tables \ref{tabla1}, \ref{tabla2} and \ref{tabla3}, whereas the orders of convergence are perfectly attained.  Finally, in Figure \ref{fig:eigen1} we present the plots of the first, third and fifth eigenfunctions associated to the pressure, and their corresponding displacement fields. 

\begin{figure}[H]
	\begin{center}
			\centering\includegraphics[height=4.5cm, width=4.2cm]{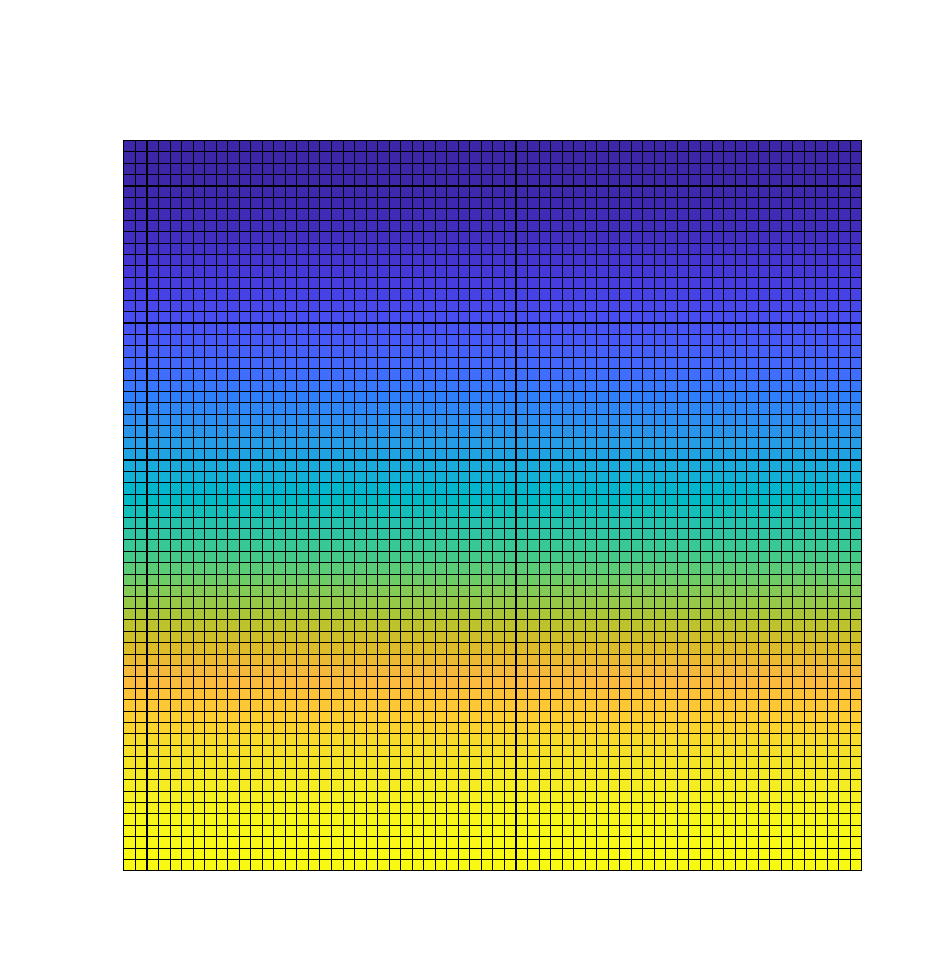}
			\centering\includegraphics[height=4.5cm, width=4.2cm]{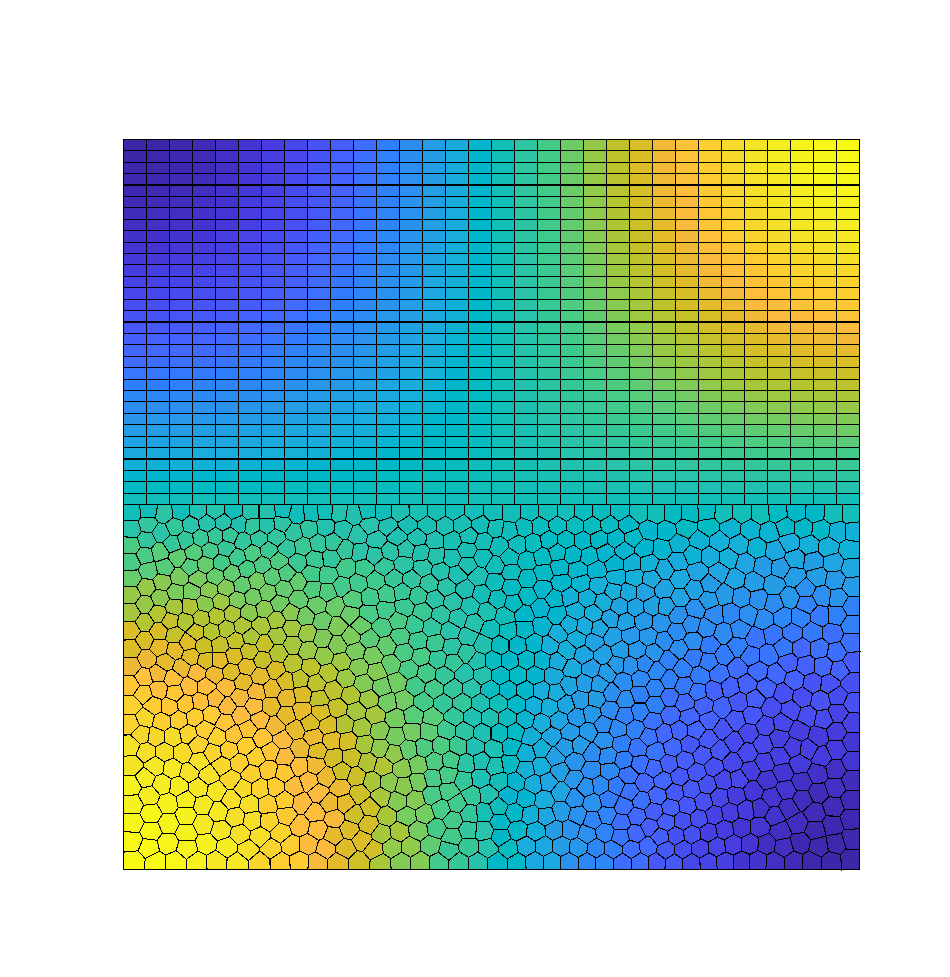}
			\centering\includegraphics[height=4.5cm, width=4.2cm]{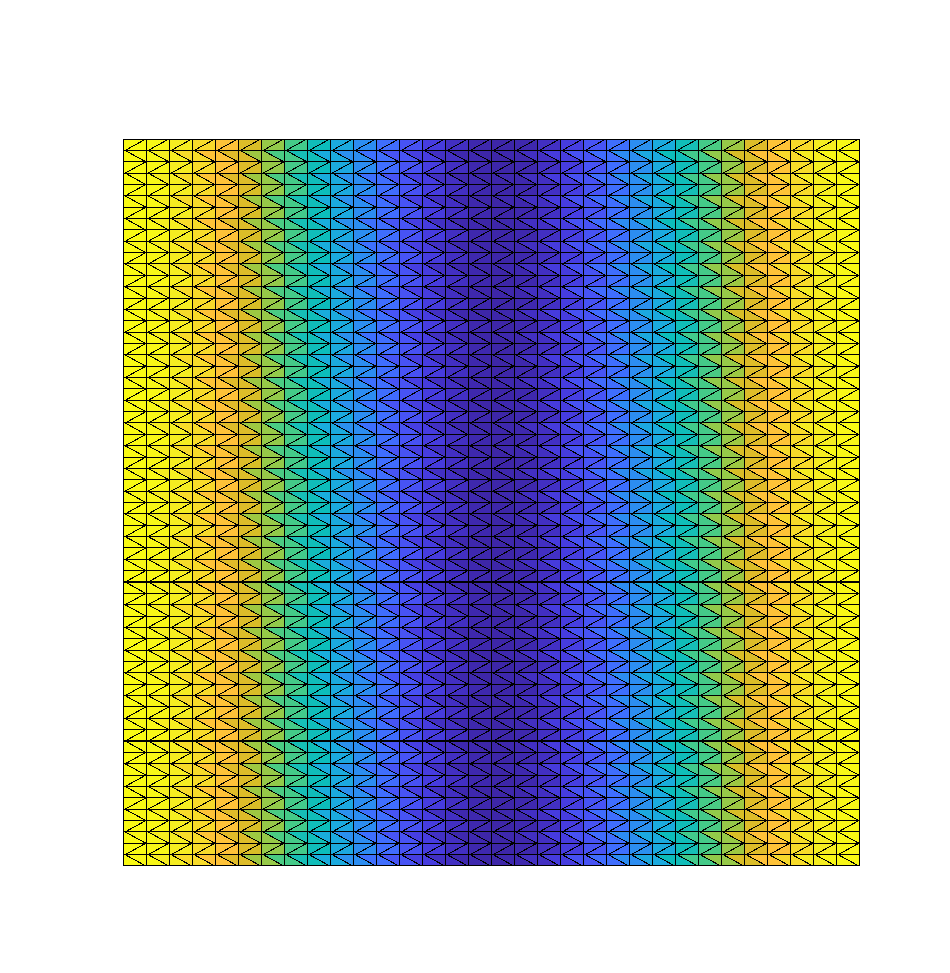} \\
		    \centering\includegraphics[height=4.5cm, width=4.2cm]{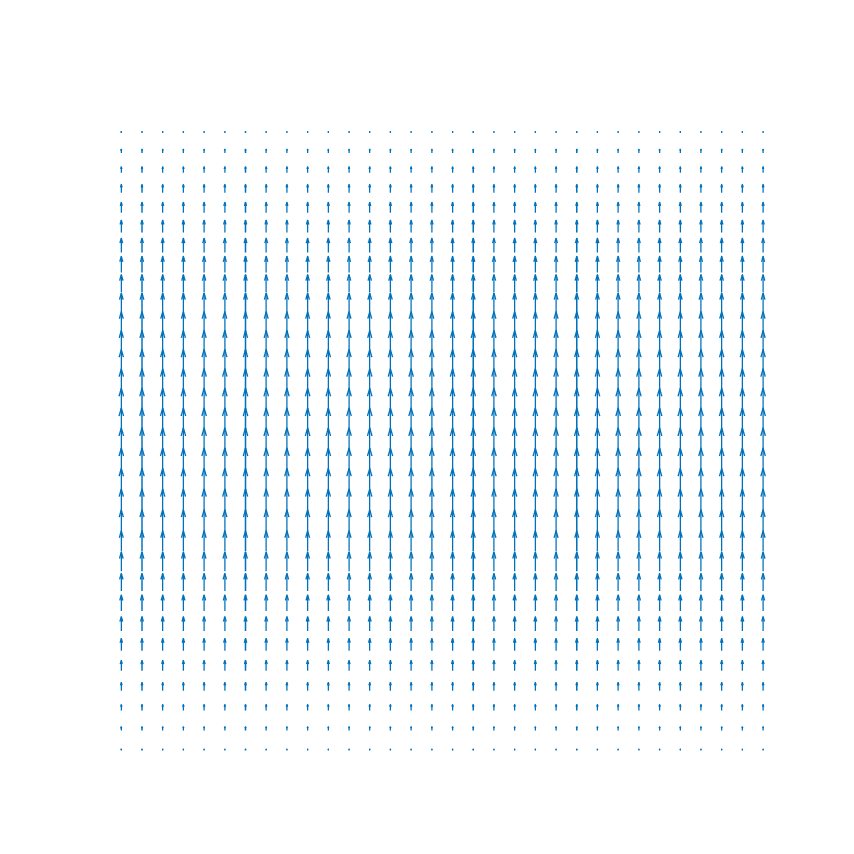}
			\centering\includegraphics[height=4.5cm, width=4.2cm]{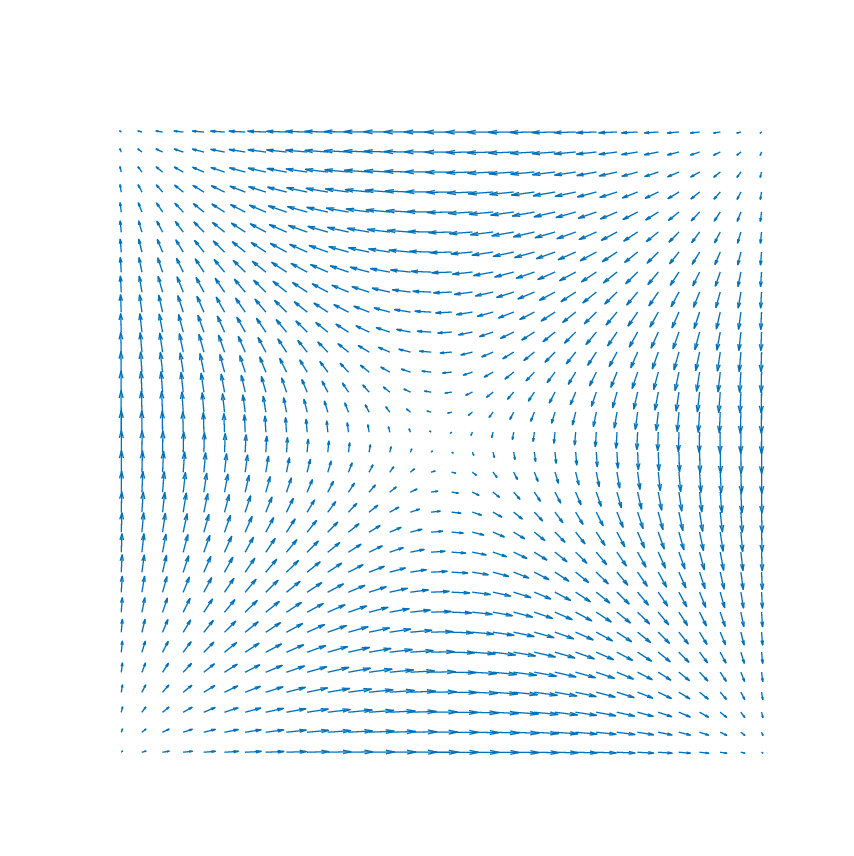}
			\centering\includegraphics[height=4.5cm, width=4.2cm]{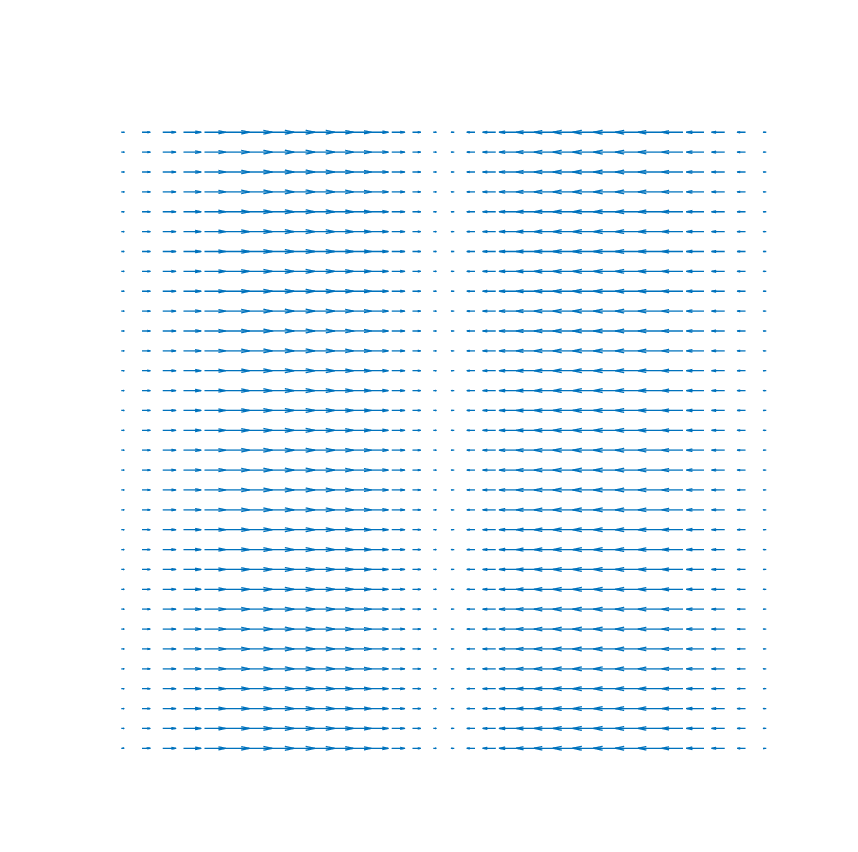}
		\caption{Plots of first, third and fifth eigenfunctions on the rectangular acoustic cavity and the corresponding displacements obtained for this test. Top row: $p_{h}^{1}$, $p_{h}^{3}$ and $p_{h}^{5}$; bottom row: corresponding displacement fields $\bu_{h}^{1}$, $\bu_{h}^{3}$ and $\bu_{h}^{5}$.}
		\label{fig:eigen1}
	\end{center}
\end{figure}

\subsection{Test 2: L-shaped domain}
Now we will consider the non-convex domain $\Omega:=(-1,1)\times(-1,1) \setminus [0,1)\times [0,1)$, which is a L-shaped domain. The boundary condition for this domain is $\nabla p\cdot\boldsymbol{n}=0$. Clearly due to the geometrical singularity of this geometrical configuration, some of the eigenfunctions of problem 
\eqref{def:acustica_pressure} result to be non sufficiently smooth and hence, a loss on the convergence order of the numerical method arises. Since for this geometry we do not have an exact solution, all our results will be compared with the extrapolated  frequencies computed by \eqref{eq:least_square}.

A sample of the  meshes that we consider for this test are reported in Figure \ref{fig:meshesxx}. Let us remark that these meshes have been obtained 
with the procedure described in \textbf{Step 1} and \textbf{Step 2}. Finally, let us mention that for this test we take physical parameters of acoustic fluids, more precisely, the ones associated to water and air. These parameter are:
\begin{itemize}
\item For water, the density is $\rho= 1000\,\text{kg}/\,\text{m}^3 $ and the sound speed $c=1430\, \text{m}/\text{s}$;
	\item For the air, the density is $\rho=1\,\text{kg}/\,\text{m}^3  $ and the sound speed $c=340\, \text{m}/\text{s}$.
\end{itemize}
 
\begin{figure}[H]
	\begin{center}
			\centering\includegraphics[height=5cm, width=5cm]{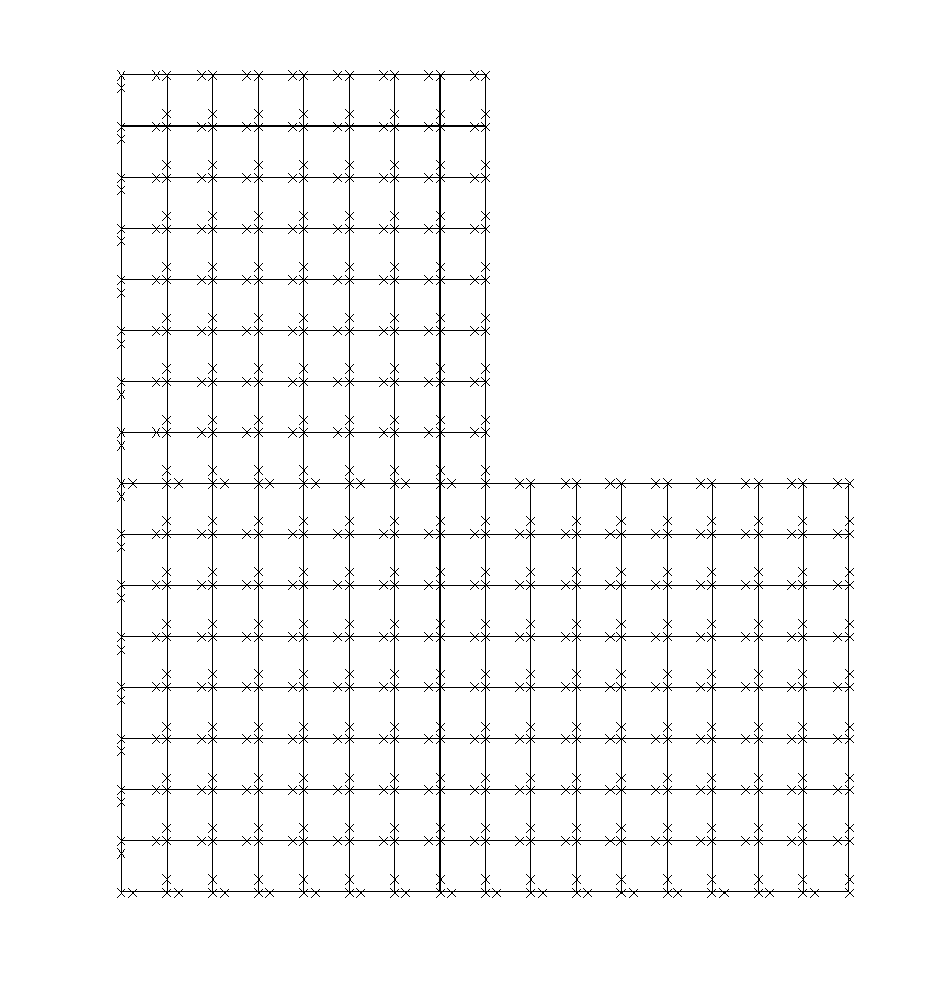}
			\centering\includegraphics[height=5cm, width=5cm]{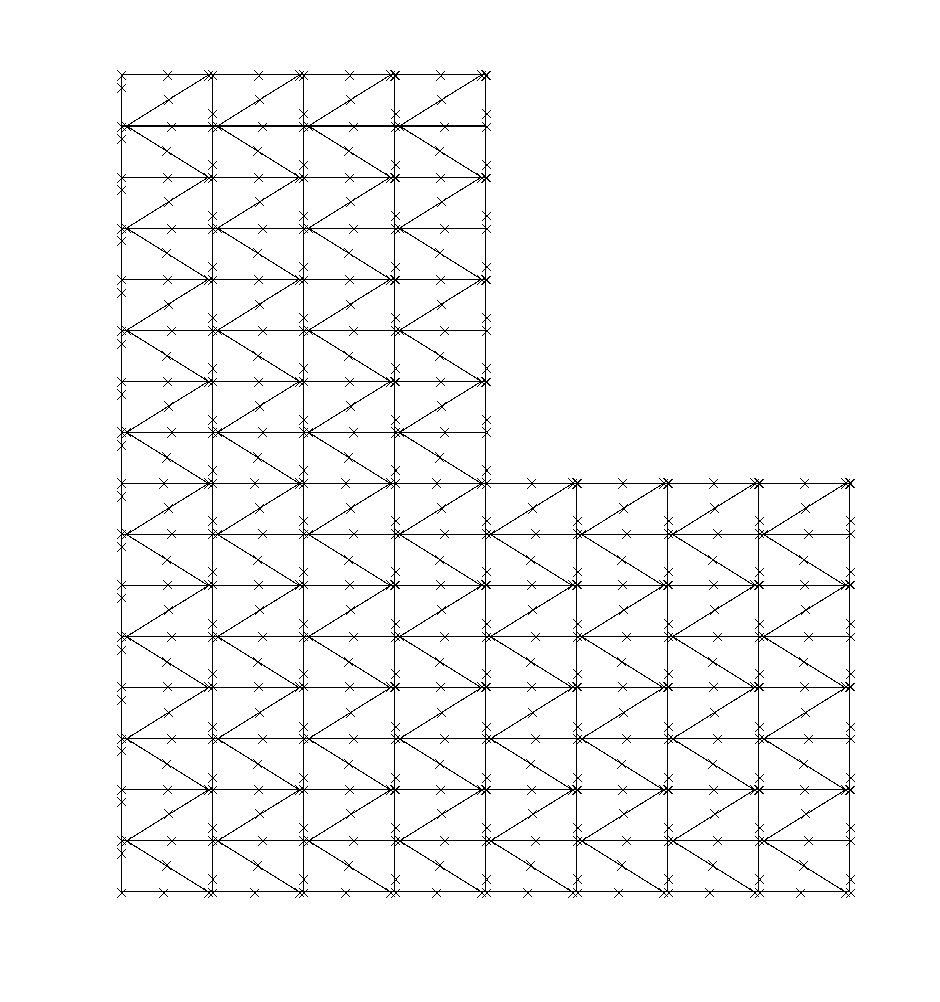}\\
		\caption{Sample of the meshes for the L-shaped domain for mesh refinement $\mathrm{N} = 16$. Left $\mathcal{T}_{h}^{4}$; right $\mathcal{T}_{h}^{5}$}
		\label{fig:meshesxx}
	\end{center}
\end{figure}
For this test, we only are concentrated  on the behavior of our method when the theoretical stabilization term defined in \eqref{stabspec} is considered, whereas the stabilization parameter $\sigma_E$ has been chosen by $\sigma_{E} := \text{tr}(a^{E}(\boldsymbol{\Pi}_{k,E},\boldsymbol{\Pi}_{k,E}))/2$. In Tables \ref{tabla7}, \ref{tabla8}, \ref{tabla9} and \ref{tabla10} we report approximated values of each one of the frequencies $\omega_{i} = \sqrt{\lambda_{i}-1}$, $i = 1,\ldots,5$, their respective orders of convergence and extrapolated frequencies for different meshes presented in Figure \ref{fig:meshesxx}. Let us recall  that the extrapolated values have been obtained with the aid of \eqref{eq:least_square} and are reported on the last column of the forthcoming tables. 

In Tables \ref{tabla7} and \ref{tabla8} we compare the results obtained for mesh $\CT_h^{4}$ when we consider the density and sound speed of water, whereas in Tables \ref{tabla9} and \ref{tabla10} we compare the results obtained for mesh $\CT_h^5$ when we consider the density and sound speed of the air. 
As in the previous test, in the row "Ratio" we report the measure of the ratios in each refinement, in order to present the presence of real small edges on the meshes.
\begin{table}[H]
\caption{Five lowest approximated frequencies, orders of convergence, extrapolated frequencies and ratios, computed with $\mathcal{T}_{h}^{4}$ and the stabilization term defined in \eqref{stabspec} for $\rho= 1000\,\text{kg}/\,\text{m}^3 $ and $c=1430\, \text{m}/\text{s}$.}
\label{tabla7}
\begin{center}
\resizebox{13cm}{!}{
\begin{tabular}{|c|c|c|c|c|c|c|c|} \hline
$\omega_{hi}$ & N = 16 & N = 32 & N = 64 & N = 128 & Order & Extrap. \\ \hline 
 $\omega_{h1}$ & 218.798 & 215.779 & 214.585 & 214.112 & 1.34 & 213.806 \\
$\omega_{h2}$ & 515.993 & 513.062 & 512.303 & 512.109 & 1.95 & 512.039 \\ 
$\omega_{h3}$ & 1448.270 & 1434.572 & 1431.148 & 1430.287 & 2.00 & 1430.004 \\
$\omega_{h4}$ & 1448.357 & 1434.590 & 1431.148 & 1430.287 & 2.00 & 1430.001 \\
$\omega_{h5}$ & 1697.359 & 1662.003 & 1653.164 & 1650.952 & 2.00 & 1650.217 \\
\hline
Ratio & 7.8740e-03 & 1.9569e-03 &  4.8852e-04  & 1.2209e-04 \\ \cline{1-5}
\end{tabular}}
\end{center}
\end{table}
\begin{table}[H]
\caption{Five lowest approximated frequencies, orders of convergence, extrapolated frequencies and ratios, computed with $\mathcal{T}_{h}^{5}$ and the stabilization term defined in \eqref{stabspec} for $\rho= 1000\,\text{kg}/\,\text{m}^3 $ and $c=1430\, \text{m}/\text{s}$.}
\label{tabla8}
\begin{center}
\resizebox{13cm}{!}{
\begin{tabular}{|c|c|c|c|c|c|c|c|} \hline
$\omega_{hi}$ & N = 16 & N = 32 & N = 64 & N = 128 & Order & Extrap. \\ \hline 
 $\omega_{h1}$ & 220.093 & 216.184 & 214.729 & 214.162 & 1.41 & 213.833 \\
$\omega_{h2}$ & 520.743 & 513.702 & 512.461 & 512.149 & 2.44 & 512.126 \\
$\omega_{h3}$ & 1456.979 & 1435.837 & 1431.136 & 1430.285 & 2.21 & 1429.959 \\
$\omega_{h4}$ & 1497.770 & 1448.411 & 1434.398 & 1431.103 & 1.86 & 1429.493 \\
$\omega_{h5}$ & 1715.697 & 1666.822 & 16547.520 & 1651.357 & 1.99 & 1650.444 \\
\hline
Ratio & 1.9124e-05 &  4.7522e-06 &  1.1862e-06  & 2.9645e-07 \\ \cline{1-5}
\end{tabular}}
\end{center}
\end{table}
\begin{table}[H]
\caption{Five lowest approximated frequencies, orders of convergence, extrapolated frequencies and ratios, computed with $\mathcal{T}_{h}^{4}$ and the stabilization term defined in \eqref{stabspec} for $\rho= 1\,\text{kg}/\,\text{m}^3 $ and $c=340\, \text{m}/\text{s}$.}
\label{tabla9}
\begin{center}
\resizebox{13cm}{!}{
\begin{tabular}{|c|c|c|c|c|c|c|c|} \hline
$\omega_{hi}$ & N = 16 & N = 32 & N = 64 & N = 128 & Order & Extrap. \\ \hline 
 $\omega_{h1}$ & 52.019 & 51.312 & 51.020 & 50.908 & 1.30 & 50.826 \\
$\omega_{h2}$ & 122.678 & 121.977 & 121.806 & 121.760 & 2.02 & 121.748 \\
$\omega_{h3}$ & 344.347 & 340.819 & 340.273 & 340.068 & 2.51 & 340.087 \\
$\omega_{h4}$ & 344.368 & 341.154 & 340.273 & 340.068 & 1.91 & 339.977 \\
$\omega_{h5}$ & 403.487 & 395.110 & 393.060 & 392.534 & 2.02 & 392.374 \\
\hline
Ratio & 7.8740e-03 & 1.9569e-03 &  4.8852e-04  & 1.2209e-04 \\ \cline{1-5}
\end{tabular}}
\end{center}
\end{table}

\begin{table}[H]
\caption{Five lowest approximated frequencies, orders of convergence, extrapolated frequencies and ratios, computed with $\mathcal{T}_{h}^{5}$ and the stabilization term defined in \eqref{stabspec} for $\rho= 1\,\text{kg}/\,\text{m}^3 $ and $c=340\, \text{m}/\text{s}$.}
\label{tabla10}
\begin{center}
\resizebox{13cm}{!}{
\begin{tabular}{|c|c|c|c|c|c|c|c|} \hline
$\omega_{hi}$ & N = 16 & N = 32 & N = 64 & N = 128 & Order & Extrap. \\ \hline 
 $\omega_{h1}$ & 52.347 & 51.407 & 51.054 & 50.920 & 1.41 & 50.840 \\
$\omega_{h2}$ & 123.234 & 122.125 & 121.844 & 121.770 & 1.97 & 121.746 \\
$\omega_{h3}$ & 344.188 & 341.004 & 340.270 & 340.068 & 2.08 & 340.023 \\
$\omega_{h4}$ & 356.418 & 344.117 & 341.046 & 340.262 & 2.00 & 340.014 \\
$\omega_{h5}$ & 408.457 & 396.474 & 393.438 & 392.630 & 1.97 & 392.374  \\
\hline
Ratio & 1.9124e-05 &  4.7522e-06 &  1.1862e-06  & 2.9645e-07 \\ \cline{1-5}
\end{tabular}}
\end{center}
\end{table}

Observe that there is no significant difference with the results obtained for meshes $\CT_h^4$ and $\CT_h^5$ for the density and sound speed of the water. The frequencies are well captured and the orders of convergence are the expected according to the geometry of the domain. The same occurs for the physical parameters of the air. Finally, we present in Figure \ref{fig:eigen2} the plots of the first, second and fifth eigenfunctions obtained in this test for the physical parameters of the water.

\begin{figure}[H]
	\begin{center}
			\centering\includegraphics[height=4.5cm, width=4.2cm]{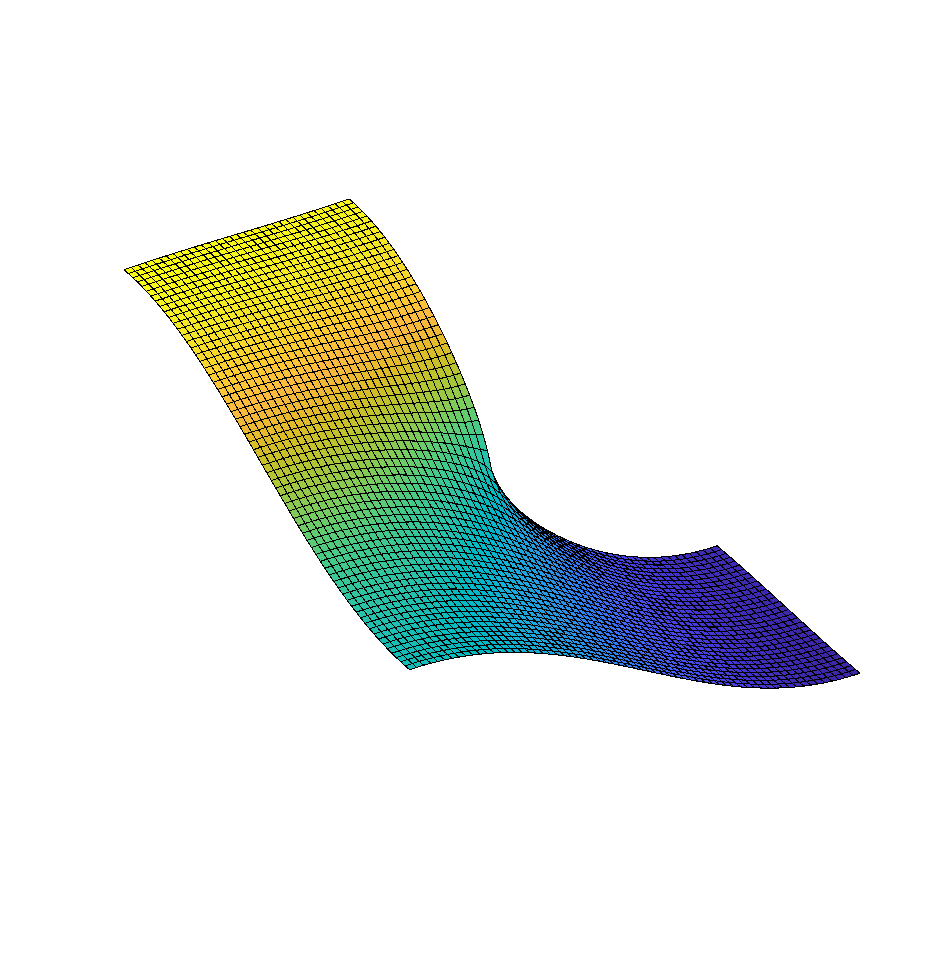}
			\centering\includegraphics[height=4.5cm, width=4.2cm]{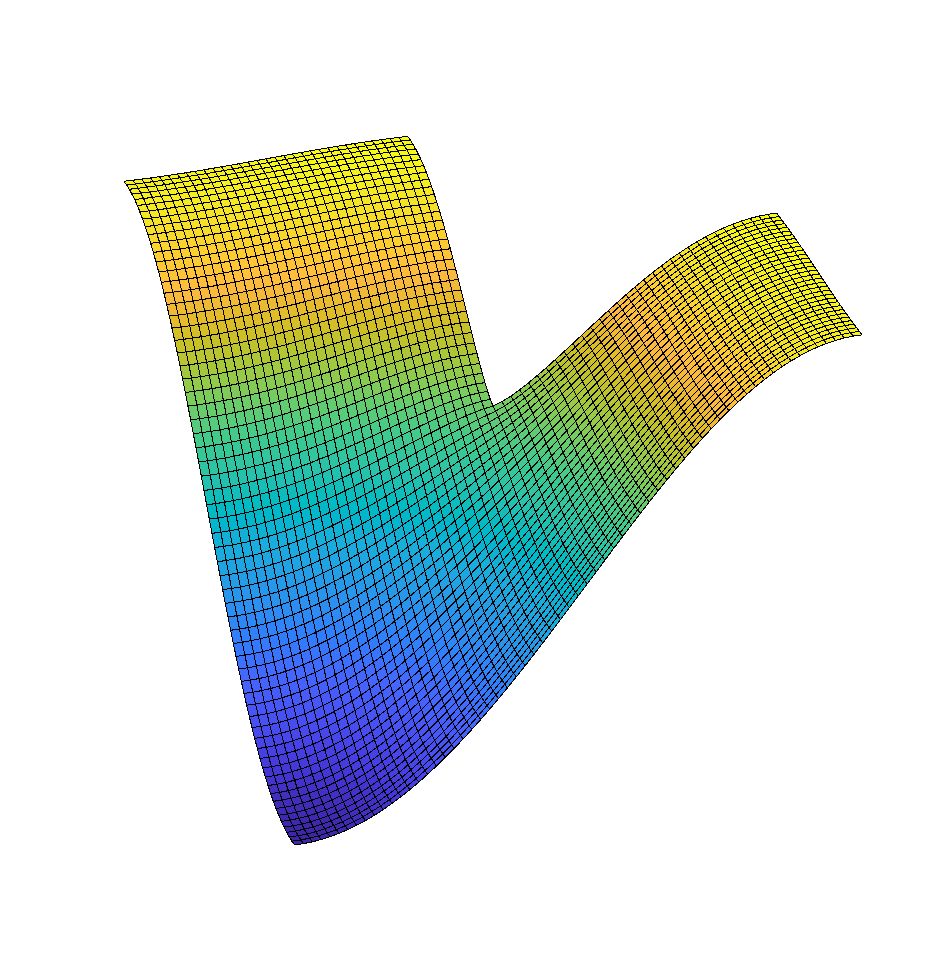}
			\centering\includegraphics[height=4.5cm, width=4.2cm]{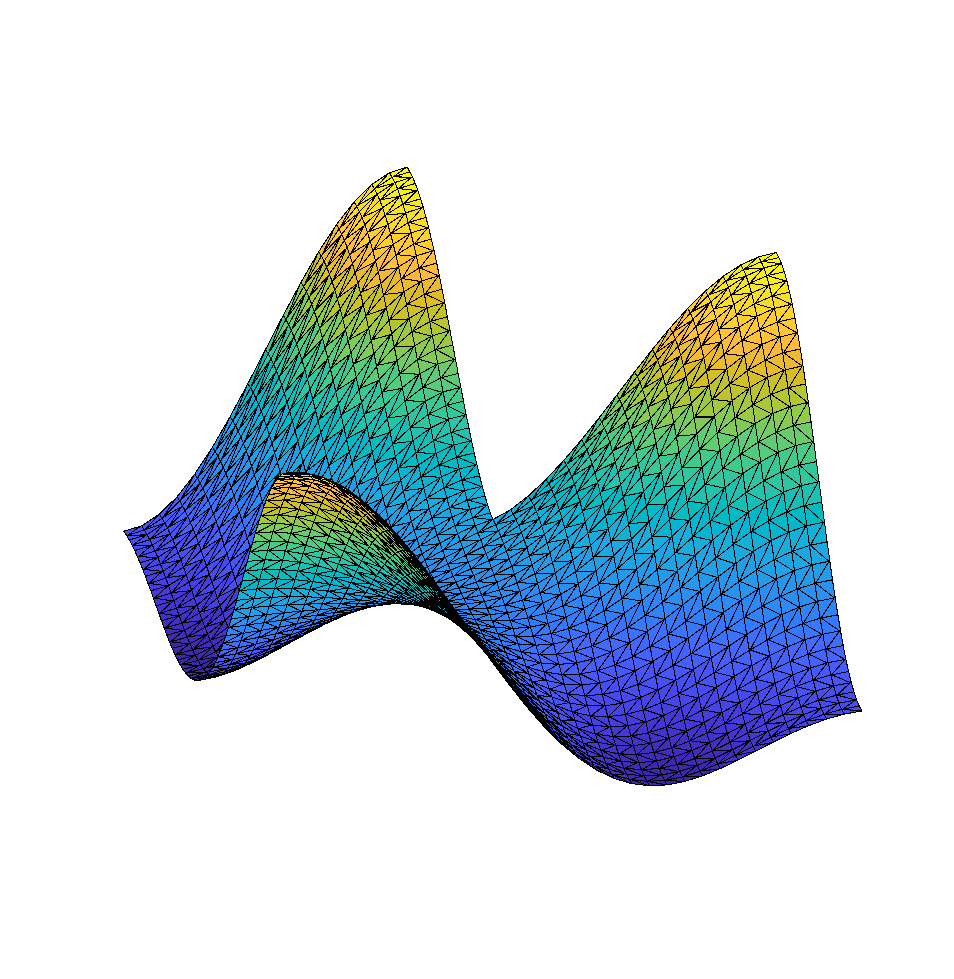} \\
					    \centering\includegraphics[height=4.5cm, width=4.2cm]{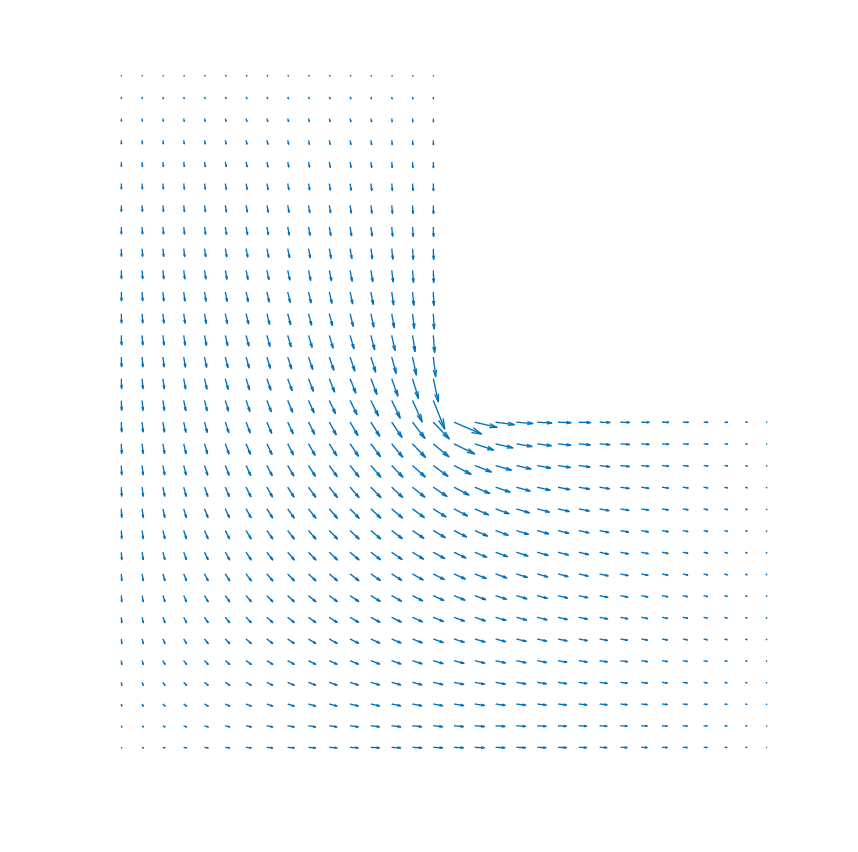}
			\centering\includegraphics[height=4.5cm, width=4.2cm]{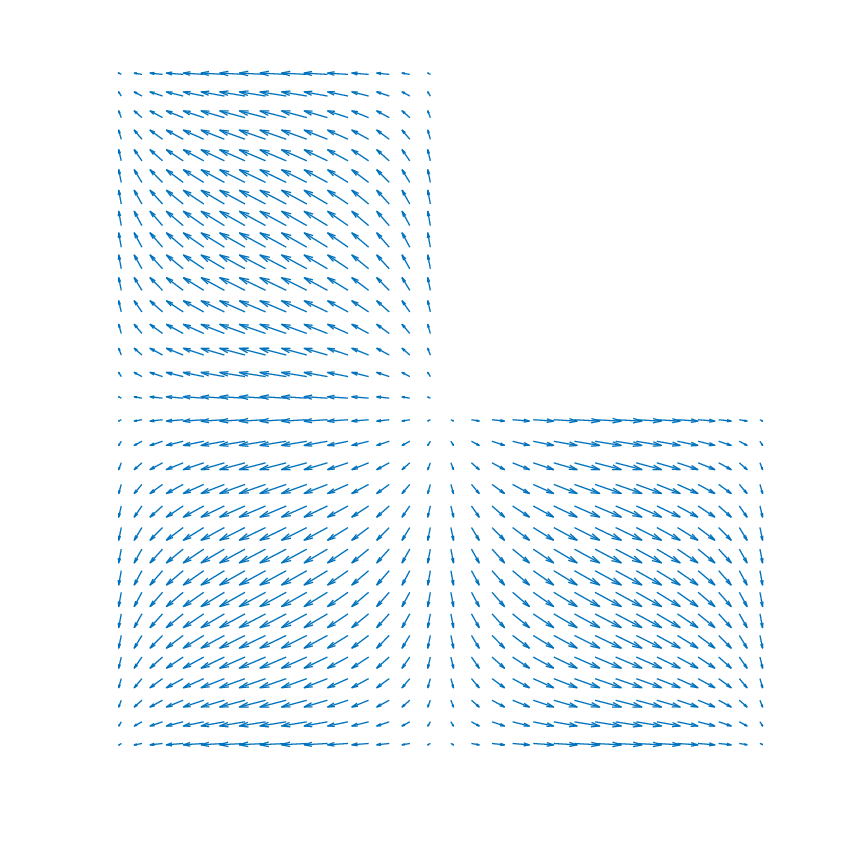}
			\centering\includegraphics[height=4.5cm, width=4.2cm]{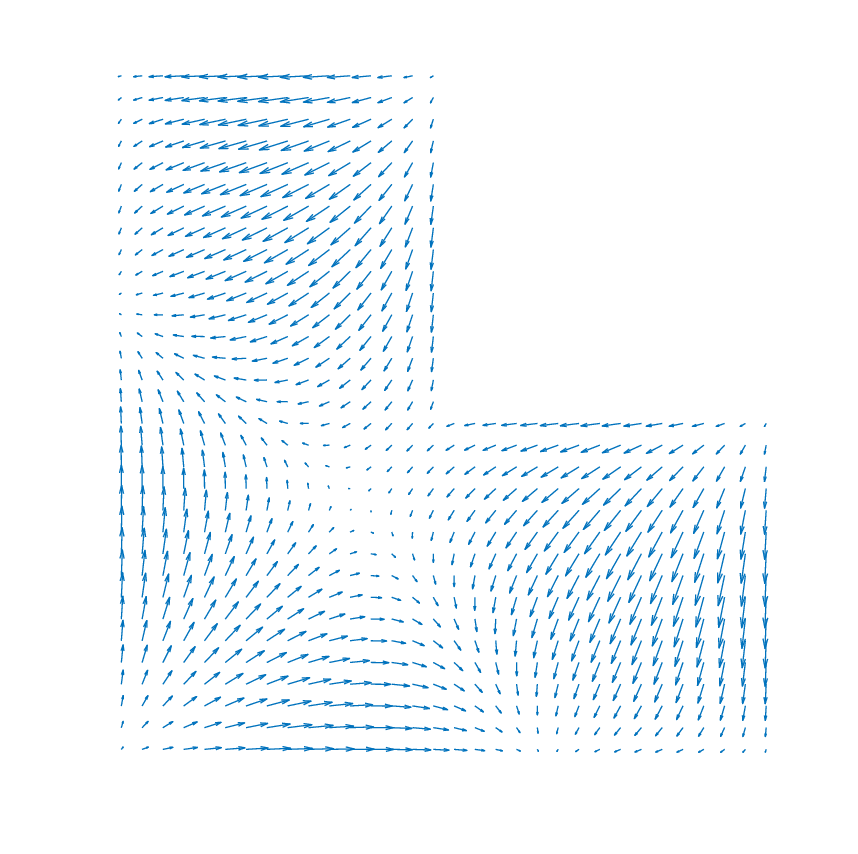}
		\caption{Plots of the first, third and fifth eigenfunctions on the L-shaped domain for the acoustic problem considering the physical parameters of  water.
}
		\label{fig:eigen2}
	\end{center}
\end{figure}

\subsection{Test 3: Effects on the stability constants}

The aim of this test is to analyze the influence of the stability constant $\sigma_E$ on the computed spectrum, in order to investigate the effects of this parameter when the spectrum is computed since for other VEMs applied to eigenvalue problems, this parameter may introduce spurious frequencies when it is not correctly determined (see \cite{danilo_eigen, MR4284360, MR4253143, MR4229296} where this phenomenon  is well documented). 

For this test, we have consider the domain $\O=(0,a)\times (0,b)$ presented in Test 1, for $a=1$ and $b=1.1$ and  the mesh $\CT_h^{2}$ presented in Figure \ref{fig:meshesx}. We observe that in this case, there are  not presence of spurious frequencies  for any choice of the stabilization parameter, which implies that the spectrum is correctly captured regardless of the stabilization parameter. So, taking this in consideration, we present in Table \ref{tabla11} the lowest three approximated frequencies, orders of convergence and extrapolated frequencies obtained for different stabilization parameters. Also, the ratios obtained for this test have been presented in the row "Ratio"  at the bottom of the table.  For the other meshes the results also hold, in the sense that no spurious frequencies appear on the computed spectrum.

The ratios reported in Table \ref{tabla11} clearly confirm the presence of small edges on the computation of the orders of convergence. Let us remark that for other type of meshes the results are similar. Let us emphasizes that the results on Table \ref{tabla11} have been obtained when the density and the sound speed are equal to one. Therefore, if we change the 
parameters to real ones as water, air, oil, etc., the behavior may be different since the configuration of the problem changes and hence,  the parameter $\sigma_E$ must be chosen in order to scale correctly like the bilinear forms considered.

\begin{table}[H]
\caption{\label{tabla11} Lowest three approximated frequencies and orders of convergence for $\mathcal{T}_{h}^{1}$ and $\rho = 1 = c$ for $4^{-2} \leq \sigma_E \leq 4^{2}$.}
\begin{center}
\begin{tabular}{|c|c|c|c|c|c|c|c|} \hline
$\sigma_E$ & $\omega_{hi}$ & $\mathrm{N} = 11$ & $\mathrm{N} = 20$ & $\mathrm{N} = 39$ & $\mathrm{N} = 88$ & Order & Extrap. \\ \hline \hline
\multirow{3}{*}{$4^{-2}$} & $\omega_{h1}$ & 0.83655 & 0.82905 & 0.82711 & 0.82661 & 1.95 & 0.82643 \\
& $\omega_{h2}$ & 1.01193 & 1.00309 & 1.00080 & 1.00020 & 1.95 & 1.00000 \\
& $\omega_{h3}$ & 1.83411 & 1.82818 & 1.82699 & 1.82658 & 2.20 & 1.82656 \\ \hline \hline
\multirow{3}{*}{$4^{-1}$} & $\omega_{h1}$ & 0.83690 & 0.82909 & 0.82711 & 0.82661 & 1.98 & 0.82644 \\
& $\omega_{h2}$ & 1.01261 & 1.00320 & 1.00080 & 1.00020 & 1.98 & 1.00000 \\
& $\omega_{h3}$ & 1.86156 & 1.83522 & 1.82864 & 1.82699 & 2.00 & 1.82644 \\ \hline \hline
\multirow{3}{*}{$4^{0}$} & $\omega_{h1}$ & 0.83705 & 0.82910 & 0.82711 & 0.82661 & 2.00 & 0.82645 \\
& $\omega_{h2}$ & 1.01282 & 1.00321 & 1.00080 & 1.00020 & 2.00 & 1.00000 \\
& $\omega_{h3}$ & 1.96986 & 1.86172 & 1.83523 & 1.82864 & 2.03 & 1.82660 \\ \hline \hline
\multirow{3}{*}{$4$} & $\omega_{h1}$ & 0.83691 & 0.82910 & 0.82711 & 0.82660 & 1.97 & 0.82643 \\
& $\omega_{h2}$ & 1.01288 & 1.00321 & 1.00080 & 1.00017 & 1.99 & 0.99997 \\
& $\omega_{h3}$ & 2.40179 & 1.96768 & 1.86159 & 1.83520 & 2.03 & 1.82693 \\ \hline \hline
\multirow{3}{*}{$4^{2}$} & $\omega_{h1}$ & 0.83711 & 0.82910 & 0.82711 & 0.82661 & 2.01 & 0.82645 \\
& $\omega_{h2}$ & 1.01290 & 1.00322 & 1.00080 & 1.00020 & 2.00 & 1.00000 \\
& $\omega_{h3}$ & 3.47880 & 2.39146 & 1.96702 & 1.86155 & 1.49 & 1.77214 \\ \hline
\multicolumn{2}{|c|}{Ratio} & 3.5665e-03  & 8.8874e-04  & 2.2201e-04 &  5.5490e-05  \\ \cline{1-6}
\end{tabular}
\end{center}
\end{table}

\bibliographystyle{siamplain}
\bibliography{references}

\begin{thebibliography}{10}

\bibitem{AABMR13}
{\sc B.~Ahmad, A.~Alsaedi, F.~Brezzi, L.~D. Marini, and A.~Russo}, {\em
  Equivalent projectors for virtual element methods}, Comput. Math. Appl., 66
  (2013), pp.~376--391, \url{https://doi.org/10.1016/j.camwa.2013.05.015}.

\bibitem{MR4581469}
{\sc D.~Amigo, F.~Lepe, and G.~Rivera}, {\em A virtual element method for the
  elasticity problem allowing small edges}, Calcolo, 60 (2023), pp.~Paper No.
  28, 34, \url{https://doi.org/10.1007/s10092-023-00522-8}.

\bibitem{danilo_eigen}
{\sc D.~Amigo, F.~Lepe, and G.~Rivera}, {\em A virtual element method for the
  elasticity spectral problem allowing for small edges}, J. Sci. Comput.,  (To
  appear).

\bibitem{BO}
{\sc I.~Babu\v{s}ka and J.~Osborn}, {\em Eigenvalue problems}, vol.~II of
  Handb. Numer. Anal., North-Holland, Amsterdam, 1991.

\bibitem{BBCMMR2013}
{\sc L.~Beir\~{a}o~da Veiga, F.~Brezzi, A.~Cangiani, G.~Manzini, L.~D. Marini,
  and A.~Russo}, {\em Basic principles of virtual element methods}, Math.
  Models Methods Appl. Sci., 23 (2013), pp.~199--214,
  \url{https://doi.org/10.1142/S0218202512500492}.

\bibitem{BLR2017}
{\sc L.~Beir\~{a}o~da Veiga, C.~Lovadina, and A.~Russo}, {\em Stability
  analysis for the virtual element method}, Math. Models Methods Appl. Sci., 27
  (2017), pp.~2557--2594, \url{https://doi.org/10.1142/S021820251750052X}.

\bibitem{BMRR}
{\sc L.~Beir\~{a}o~da Veiga, D.~Mora, G.~Rivera, and R.~Rodr\'{\i}guez}, {\em A
  virtual element method for the acoustic vibration problem}, Numer. Math., 136
  (2017), pp.~725--763, \url{https://doi.org/10.1007/s00211-016-0855-5}.

\bibitem{MR1770352}
{\sc A.~Berm\'{u}dez, R.~G. Dur\'{a}n, R.~Rodr\'{\i}guez, and J.~Solomin}, {\em
  Finite element analysis of a quadratic eigenvalue problem arising in
  dissipative acoustics}, SIAM J. Numer. Anal., 38 (2000), pp.~267--291,
  \url{https://doi.org/10.1137/S0036142999360160}.

\bibitem{MR1993937}
{\sc A.~Berm\'{u}dez, P.~Gamallo, L.~Hervella-Nieto, and R.~Rodr\'{\i}guez},
  {\em Finite element analysis of pressure formulation of the elastoacoustic
  problem}, Numer. Math., 95 (2003), pp.~29--51,
  \url{https://doi.org/10.1007/s00211-002-0414-0}.

\bibitem{MR2652780}
{\sc D.~Boffi}, {\em Finite element approximation of eigenvalue problems}, Acta
  Numer., 19 (2010), pp.~1--120,
  \url{https://doi.org/10.1017/S0962492910000012}.

\bibitem{BS-2008}
{\sc S.~C. Brenner and L.~R. Scott}, {\em The mathematical theory of finite
  element methods}, vol.~15 of Texts in Applied Mathematics, Springer, New
  York, third~ed., 2008, \url{https://doi.org/10.1007/978-0-387-75934-0}.

\bibitem{MR3815658}
{\sc S.~C. Brenner and L.-Y. Sung}, {\em Virtual element methods on meshes with
  small edges or faces}, Math. Models Methods Appl. Sci., 28 (2018),
  pp.~1291--1336, \url{https://doi.org/10.1142/S0218202518500355}.

\bibitem{gardini2}
{\sc O.~{\v{C}}ert{\'{\i}}k, F.~Gardini, G.~Manzini, and G.~Vacca}, {\em The
  virtual element method for eigenvalue problems with potential terms on
  polytopic meshes}, 63 (2018), pp.~333--365,
  \url{https://doi.org/10.21136/am.2018.0093-18}.

\bibitem{GMV2018}
{\sc F.~Gardini, G.~Manzini, and G.~Vacca}, {\em The nonconforming virtual
  element method for eigenvalue problems}, ESAIM Math. Model. Numer. Anal., 53
  (2019), pp.~749--774, \url{https://doi.org/10.1051/m2an/2018074}.

\bibitem{MR3867390}
{\sc F.~Gardini and G.~Vacca}, {\em Virtual element method for second-order
  elliptic eigenvalue problems}, IMA J. Numer. Anal., 38 (2018),
  pp.~2026--2054, \url{https://doi.org/10.1093/imanum/drx063}.

\bibitem{MR0775683}
{\sc P.~Grisvard}, {\em Elliptic problems in nonsmooth domains}, vol.~24 of
  Monographs and Studies in Mathematics, Pitman (Advanced Publishing Program),
  Boston, MA, 1985.

\bibitem{MR0203473}
{\sc T.~Kato}, {\em Perturbation theory for linear operators}, vol.~Band 132 of
  Grundlehren der Mathematischen Wissenschaften, Springer-Verlag, Berlin-New
  York, second~ed., 1976.

\bibitem{MR3854050}
{\sc F.~Lepe, S.~Meddahi, D.~Mora, and R.~Rodr\'{\i}guez}, {\em Acoustic
  vibration problem for dissipative fluids}, Math. Comp., 88 (2019),
  pp.~45--71, \url{https://doi.org/10.1090/mcom/3336}.

\bibitem{MR4284360}
{\sc F.~Lepe, D.~Mora, G.~Rivera, and I.~Vel\'{a}squez}, {\em A virtual element
  method for the {S}teklov eigenvalue problem allowing small edges}, J. Sci.
  Comput., 88 (2021), pp.~Paper No. 44, 21,
  \url{https://doi.org/10.1007/s10915-021-01555-3}.

\bibitem{MR4550402}
{\sc F.~Lepe, D.~Mora, G.~Rivera, and I.~Vel\'{a}squez}, {\em A posteriori
  virtual element method for the acoustic vibration problem}, Adv. Comput.
  Math., 49 (2023), pp.~Paper No. 10, 29,
  \url{https://doi.org/10.1007/s10444-022-10003-1}.

\bibitem{MR4253143}
{\sc F.~Lepe and G.~Rivera}, {\em A priori error analysis for a mixed {VEM}
  discretization of the spectral problem for the {L}aplacian operator},
  Calcolo, 58 (2021), pp.~Paper No. 20, 30,
  \url{https://doi.org/10.1007/s10092-021-00412-x}.

\bibitem{MR4229296}
{\sc F.~Lepe and G.~Rivera}, {\em A virtual element approximation for the
  pseudostress formulation of the {S}tokes eigenvalue problem}, Comput. Methods
  Appl. Mech. Engrg., 379 (2021), pp.~Paper No. 113753, 21,
  \url{https://doi.org/10.1016/j.cma.2021.113753}.

\bibitem{LR3}
{\sc F.~Lepe and G.~Rivera, G}, {\em Vem discretization allowing small edges
  for the reaction-convection-diffusion equation: source and spectral
  problems}, ESAIM Math. Model. Numer. Anal., To appear.

\bibitem{MR4050542}
{\sc D.~Mora and G.~Rivera}, {\em {\it {A} priori} and {\it a posteriori} error
  estimates for a virtual element spectral analysis for the elasticity
  equations}, IMA J. Numer. Anal., 40 (2020), pp.~322--357,
  \url{https://doi.org/10.1093/imanum/dry063}.

\bibitem{MR3340705}
{\sc D.~Mora, G.~Rivera, and R.~Rodr\'{\i}guez}, {\em A virtual element method
  for the {S}teklov eigenvalue problem}, Math. Models Methods Appl. Sci., 25
  (2015), pp.~1421--1445, \url{https://doi.org/10.1142/S0218202515500372}.

\bibitem{MVsiam2021}
{\sc D.~Mora and I.~Vel\'{a}squez}, {\em Virtual elements for the transmission
  eigenvalue problem on polytopal meshes}, SIAM J. Sci. Comput., 43 (2021),
  pp.~A2425--A2447, \url{https://doi.org/10.1137/20M1347887}.

\bibitem{Soize1997StructuralAA}
{\sc C.~Soize and R.~Ohayon}, {\em Structural acoustics and vibration:
  Mechanical models, variational formulations and discretization}, 1997.

\end{thebibliography}
\end{document}